\documentclass[12pt]{amsart}

\usepackage{amsmath,amssymb,amsxtra,anysize,bbm,stmaryrd,pinlabel}
\usepackage{tikz,tikz-cd}
\usepackage{combelow,comment}

\RequirePackage{color}
\definecolor{references}{rgb}{0,0,1}

\RequirePackage[pdftex,
colorlinks = true,
urlcolor = references, 
citecolor = references, 
linkcolor = references, 
]
{hyperref}

\newtheorem{lemma}{Lemma}[section]
\newtheorem{theorem}[lemma]{Theorem}
\newtheorem{conjecture}[lemma]{Conjecture}
\newtheorem{corollary}[lemma]{Corollary}
\newtheorem{proposition}[lemma]{Proposition}
\theoremstyle{definition}
\newtheorem{example}[lemma]{Example}
\newtheorem{remark}[lemma]{Remark}
\newtheorem{definition}[lemma]{Definition}

\newcommand{\inv}{^{-1}}
\renewcommand{\a}{\rho}
\newcommand{\one}{\mathbbm{1}}
\newcommand{\Alt}{\mathrm{Alt}}
\newcommand{\C}{\mathbb{C}}
\newcommand{\Q}{\mathbb{Q}}
\newcommand{\Z}{\mathbb{Z}}
\newcommand{\HY}{\mathrm{HY}}
\newcommand{\CY}{\mathrm{CY}}
\newcommand{\HH}{\mathrm{HH}}
\newcommand{\HHH}{\mathrm{HHH}}
\newcommand{\CHH}{\mathrm{CHH}}
\newcommand{\Br}{\mathrm{Br}}
\newcommand{\FT}{\mathrm{FT}}
\newcommand{\cJ}{\mathcal{J}}
\newcommand{\fgl}{\mathfrak{gl}}
\newcommand{\PP}{\mathbf{P}}
\newcommand{\Id}{\mathrm{Id}}
\newcommand{\id}{\Id}
\newcommand{\colim}{\operatornamewithlimits{colim}}
\newcommand{\hocolim}{\mathrm{hocolim}}
\newcommand{\sgn}{\mathrm{sgn}}
\newcommand{\cA}{\mathcal{A}}
\newcommand{\Cone}{\mathrm{Cone}}
\newcommand{\Span}{\mathrm{Span}}

\newcommand{\SBim}{\mathrm{SBim}}
\newcommand{\Frac}{\mathrm{Frac}}
\newcommand{\wt}{\mathrm{wt}}
\newcommand{\yy}{\mathbf{y}}
\newcommand{\xx}{\mathbf{x}}
\newcommand{\uu}{\mathbf{u}}
\newcommand{\ttheta}{\boldsymbol{\theta}}
\newcommand{\xxi}{\boldsymbol{\xi}}

\newcommand{\sll}{\mathfrak{sl}}
\newcommand{\gll}{\mathfrak{gl}}
\renewcommand{\b}{\beta}

\renewcommand{\k}{\mathbbm{k}}

\newcommand{\Hom}{\mathrm{Hom}}
\newcommand{\End}{\mathrm{End}}
\newcommand{\Tor}{\mathrm{Tor}}
\newcommand{\Ext}{\mathrm{Ext}}
\newcommand{\Ch}{\mathrm{Ch}}
\newcommand{\yBi}{B_i}
\newcommand{\yR}{R}
\newcommand{\yTi}{F^y(\sigma_i)}
\newcommand{\yTii}{F^y\left(\sigma_i\inv\right)}
\newcommand{\tw}{\operatorname{tw}}

\newcommand{\FTy}{\mathbf{FT}}
\newcommand{\cotimes}{\,\widehat{\otimes}\,}

\newcommand{\sT}{\operatorname{sF}}
\newcommand{\sTy}{\operatorname{sF^y}}

\newcommand{\htt}{\mathbf{ht}}
\newcommand{\ft}{\mathbf{ft}}

\title{Stable deformed $\mathfrak{gl}_N$ homology of torus knots}

\author{William Ballinger}
\address{Department of Mathematics, Harvard University
\\
1 Oxford Street
Cambridge, MA 02138}
\email{ballinger@math.harvard.edu}
\author{Eugene Gorsky}
\address{Department of Mathematics, University of California Davis\\ One Shields Avenue, Davis CA 95616}
\email{egorskiy@ucdavis.edu}
\author{Matthew Hogancamp}
\address{Department of Mathematics, Northeastern University \\
360 Huntington Ave, Boston, MA 02115}
\email{m.hogancamp@northeastern.edu}
\author{Joshua Wang}
\address{\parbox{\linewidth}{
School of Mathematics, Institute for Advanced Studies, 1 Einstein Drive, Princeton NJ 08540\\Department of Mathematics, Princeton University, Fine Hall, Washington Road, Princeton NJ 08540
}\vspace{3pt}}
\email{joshuaxw@princeton.edu}

\begin{document}

\begin{abstract}
We compute the $E_2$ page in the Rasmussen spectral sequence from triply graded to $\fgl_N$ Khovanov--Rozansky stable homology of torus knots. This confirms a weak form of the conjecture of the second author, Oblomkov, and Rasmussen. The main tool is the link-splitting deformation, or $y$-ification, of link homology; in the $y$-ified context, the relevant Rasmussen spectral sequence collapses and we explicitly compute the $y$-ified $\fgl_N$ stable Khovanov--Rozansky homology of torus knots for all $N$.
\end{abstract}

\maketitle
\thispagestyle{empty}

\section{Introduction}

The Khovanov homology groups of torus knots $T(n,m)$ are famously intricate. Although there are algorithms \cite{katlas,FoamHo,Khoca} that in principle can compute $\mathrm{Kh}(T(n,m))$ for any given $(n,m)$, the resulting computations do not exhibit clear patterns, and there is no explicit conjectural prediction for general $(n,m)$ in the literature. However, the $m\to\infty$ stable limit $\mathrm{Kh}(T(n,\infty))$ \cite{Stosic1,Stosic2} has a concise conjectural description given by the second author, Oblomkov, and Rasmussen \cite{GOR}. We revisit and partially prove this conjecture together with its $\fgl_N$ Khovanov--Rozansky analogues \cite{GL,GNR,GORS} which we now recall.

Consider the algebra $\Z[u_0,\ldots,u_{n-1}] \otimes \Lambda(\xi_0,\ldots,\xi_{n-1})$ equipped with the differential $d_N$ satisfying \begin{equation}\label{eq:dNdifferential}
    d_N(u_k) = 0 \qquad\qquad d_N(\xi_k) = \sum_{\substack{0 \leq i_1,\ldots,i_N < n\\ i_1 + \cdots + i_N = k}} u_{i_1}\cdots u_{i_N}
\end{equation}extended by the graded Leibniz rule. This complex is just the Koszul complex associated with the non-regular sequence of polynomials $d_N(\xi_k)$. When $u_k$ and $\xi_k$ are equipped with the quantum and cohomological bigradings $(2k + 2, -2k)$ and $(2k + 2N, -2k-1)$ respectively, the homology of this complex is conjecturally isomorphic to the stable $\fgl_N$ Khovanov--Rozansky homology $H_{\fgl_N}(T(n,\infty);\Z)$ \cite[Conjecture 1.8]{GORS}, \cite[Conjecture 1]{GL}.

\begin{theorem}\label{thm:page1theorem}
    For $N,n\ge 1$, there is a spectral sequence with $\Q$ coefficients that abuts to $H_{\fgl_N}(T(n,\infty);\Q)$ whose $E_1$ page is $\Q[u_0,\ldots,u_{n-1}] \otimes \Lambda(\xi_0,\ldots,\xi_{n-1})$ equipped with the differential $d_N$ defined by (\ref{eq:dNdifferential}). 
    
    Furthermore, the spectral sequence is multiplicative, so in particular, the differential on every page of the spectral sequence satisfies the graded Leibniz rule.
\end{theorem}

\begin{remark}
Theorem \ref{thm:page1theorem} is a vast improvement of a result in \cite{HogPoly}, which computes the $E_1$ page without precise knowledge of its differential, and only in the case $N=2$.
\end{remark}
\begin{remark}
    The conjectural description of $H_{\fgl_N}(T(n,\infty);
    \Q)$ of \cite{GL, GOR,GORS} is the $E_2$ page of the spectral sequence of Theorem~\ref{thm:page1theorem}, so the conjecture over $\Q$ is equivalent to the collapse of this spectral sequence at the $E_2$ page.
\end{remark}
\begin{remark}
    Collapse of the spectral sequence of Theorem~\ref{thm:page1theorem} at the $E_2$ page is implied by \cite[Conjecture 4]{GL} restated in this paper as Conjecture~\ref{conj: mu}. This conjecture provides a collection of elements in the homology of $d_N$ that conjecturally generate the homology as an algebra. For degree reasons, all higher differentials of the spectral sequence vanish on these elements, so the graded Leibnitz rule implies that all higher differentials vanish if these elements do indeed generate the $E_2$ page. For more details, see section~\ref{sec: algebraic conjecture}. 
\end{remark}

\begin{remark}
We expect a similar spectral sequence over any field. However, the methods of our proof rely on the main result of \cite{GH} which requires characteristic zero.
\end{remark}

\subsection{Motivation: $S^n$-colored link homology}

Given a semisimple Lie algebra $\mathfrak{g}$, or rather the corresponding quantum group $U_q(\mathfrak{g})$, and a representation $V$, classical work of Reshetikhin and Turaev \cite{RT} defines a knot invariant. To a given knot $K$, it assigns a polynomial $P_{\mathfrak{g},V}(K;q)$ in a single variable $q$ that is a topological invariant of $K$, usually called the Reshetikhin-Turaev invariant of $K$ colored by $V$.

For the case where $\mathfrak{g}=\mathfrak{gl}_N$ with $V=V_{\lambda;N}$, the irreducible representation corresponding to a partition $\lambda$, the Reshetikhin-Turaev invariants can be packaged into the two variable colored HOMFLY polynomial $P_{\lambda}(K;a,q)$ satisfying
\begin{equation}
\label{eq: colored HOMFLY}
P_{\mathfrak{gl}_N,V_{\lambda;N}}(K;q)=P_{\lambda}(K;a=q^N,q).
\end{equation}

In recent decades, there has been extensive work on categorifying the above colored invariants by colored link homology theories. For $\lambda=(1)$ the corresponding ``uncolored" $\mathfrak{gl}_N$ homology was constructed by Khovanov and Rozansky in \cite{Kho,KR1}, and HOMFLY homology was constructed in \cite{KhSoergel,KR2}. 
This approach was generalized to the case $\lambda=(1^n)$, or the so-called $\wedge^n$ colored homology in \cite{Wu,Yonezawa} and can be described using the formalism of webs and foams \cite{QR,RW}. The $a = q^N$ specialization of (\ref{eq: colored HOMFLY}) is categorified by Rasmussen's spectral sequence from HOMFLY homology to $\fgl_N$ homology \cite{RasDiff}.

One approach to defining $S^n$-colored homology, which corresponds to $\lambda = (n)$, is as follows. Given a knot $K$ and a link homology theory $\mathcal{H}$, one considers the sequence of cables $K_{n,mn}, m\ge 0$ and defines the colimit
\begin{equation}
\label{eq: colimit intro}
\mathcal{H}_{S^n}(K):=\varinjlim\left[
\mathcal{H}(K_{n,0})\xrightarrow{\a_n}
\mathcal{H}(K_{n,n})\xrightarrow{\a_n}
\mathcal{H}(K_{n,2n})\xrightarrow{\a_n}\cdots
\right].
\end{equation}
Here $\mathcal{H}$ can refer to various flavors of Khovanov and Khovanov--Rozansky homology which we review below, and the connecting maps $\a_n:\mathcal{H}(K_{n,mn})\rightarrow \mathcal{H}(K_{n,(m+1)n})$ need to be carefully constructed. 
In all cases of interest, it is known that the colimit \eqref{eq: colimit intro} is finite-dimensional in each homogeneous degree and its graded Euler characteristic agrees with the $S^n$-colored link polynomial. 
Thus, one can refer to $\mathcal{H}_{S^n}(K)$ as the $S^n$-colored homology of $K$.
See Section \ref{sec: colored homology} for more details and references. 

Computing $\mathcal{H}_{S^n}(K)$ remains a major open problem. When $K=O_1$ is the unknot, the limit \eqref{eq: colimit intro} becomes
\begin{equation}
\label{eq: colimit intro 2}
\mathcal{H}_{S^n}(O_1)=\mathcal{H}(T(n,\infty)):=\varinjlim\left[
\mathcal{H}(T(n,0))\xrightarrow{\a_n}
\mathcal{H}(T(n,n))\xrightarrow{\a_n}
\mathcal{H}(T(n,2n))\xrightarrow{\a_n}\cdots
\right]
\end{equation}
and coincides with the stable homology of $n$-strand torus links $T(n,m)$ at $m\to \infty$
\cite{Stosic1,Stosic2}.  
In particular, the $S^n$-colored $\fgl_N$ homology of the unknot coincides with the stable $\fgl_N$ homology of $T(n,\infty)$, \[
    H_{\fgl_N,S^n}(O_1) = H_{\fgl_N}(T(n,\infty)).
\]Theorem~\ref{thm:page1theorem} may be viewed as progress towards computing the $S^n$-colored $\fgl_N$ homology of the unknot. We note that the invariant of the unknot plays an important role because 
$\mathcal{H}_{S^n}(O_1)$ has a natural algebra structure and $\mathcal{H}_{S^n}(K)$ is a module over this algebra for any knot $K$, for all link homologies in this paper.
For experts, we remark that $\mathcal{H}_{S^n}(K)$ might be called ``projector-colored homology" which is different from several other models for colored Khovanov homology, see \cite{RvM} and references therein. We do not discuss other models or make any computations for them in this paper.

The spectral sequence in Theorem~\ref{thm:page1theorem} is an extension of Rasmussen's spectral sequence from triply graded homology to $\fgl_N$ homology. The spectral sequence abuts to the $S^n$-colored $\fgl_N$ homology of the unknot, and the $E_1$ page is the $S^n$-colored triply graded homology of the unknot, which is computed in \cite{Hog} and stated below. See Section \ref{ss:Pn} for more details.

\begin{theorem}[\cite{Hog}]
\label{thm: HHH intro}
a) For triply graded Khovanov--Rozansky homology $\HHH$ with coefficients in $\Z$, one has
$$
\HHH_{S^n}(O_1;\Z)=\HHH(T(n,\infty);\Z)=\Z[u_0,\ldots,u_{n-1},\xi_0,\ldots,\xi_{n-1}].
$$
Here the variables $u_0,\ldots,u_{n-1}$ are even and $\xi_0,\ldots,\xi_{n-1}$ are odd, and their degrees are given by Theorem \ref{thm:end Pn}.

b) $\HHH(T(n,\infty);\Z)$ is naturally a module over the triply graded homology of the $n$-component unlink
$$
\HHH(O_n;\Z)=\Z[x_1,\ldots,x_n,\theta_1,\ldots,\theta_n]
$$
where all even variables $x_i$ act by $u_0$ and all odd variables $\theta_i$ act by $\xi_0$.
\end{theorem}

\subsection{$y$-ification}

To prove Theorem \ref{thm:page1theorem}, we use a deformation, called $y$-ification, of triply graded homology developed in \cite{CK,GH} and denoted by $\HY(L)$. The deformation parameters are denoted by $y_1,\ldots,y_n$.  Throughout this section we work over a characteristic zero field $\k$.

Our second main result describes the stable $y$-fied HOMFLY homology of $T(n,\infty)$ or, equivalently, the $S^n$-colored $y$-ified homology of the unknot.

\begin{theorem}
\label{thm: yified stable homfly intro}
Let $\k$ be a field of characteristic zero.

a) The stable $y$-fied HOMFLY homology of $T(n,\infty)$ is isomorphic to the  free polynomial algebra
$$
\HY_{S^n}(O_1;\k)=\HY(T(n,\infty);\k)\simeq \k[y_1,\ldots,y_n,u_0,\ldots,u_{n-1},\xi_0,\dots,\xi_{n-1}].
$$
Here the variables $y_1,\ldots,y_n,u_0,\ldots,u_{n-1}$ are even, $\xi_0,\dots,\xi_{n-1}$ are odd, and their degrees are given in Theorem \ref{thm: y-ified projector}.

b) $\HY(T(n,\infty);\k)$ is naturally a module over the $y$-ified HOMFLY homology of the $n$-component unlink
$$
\HY(O_n;\k)=\k[x_1,\ldots,x_n,y_1,\ldots,y_n,\theta_1,\ldots,\theta_n]
$$
where $y_i$ acts by multiplication and $x_i,\theta_i$ act by
\begin{equation}
\label{eq: interpolation intro}
x_i=u_0+u_1y_i+\ldots+u_{n-1}y_{i}^{n-1},\quad 
\theta_i=\xi_0+\xi_1y_i+\ldots+\xi_{n-1}y_{i}^{n-1}.\\
\end{equation}
c) $\HY(T(n,\infty);\k)$ has a natural action of the symmetric group $S_n$ which simultaneously permutes $x_i,y_i$, and $\theta_i$ and fixes $u_k$ and $\xi_k$.
\end{theorem}

\begin{remark}
One can interpret \eqref{eq: interpolation intro} as an interpolation problem. Indeed, the variables $u_k$ (resp. $\xi_k$) yield one-variable polynomials
$$
u(z)=u_0+u_1z+\ldots+u_{n-1}z^{n-1},\quad 
\xi(z)=\xi_0+\xi_1z+\ldots+\xi_{n-1}z^{n-1}.\\
$$
and  \eqref{eq: interpolation intro} says that the values of these polynomials at $z=y_i$ are determined by 
$$
u(z=y_i)=x_i,\ \xi(z=y_i)=\theta_i.
$$
\end{remark}

Part (a) of Theorem \ref{thm: yified stable homfly intro} can be easily deduced from Theorem \ref{thm: HHH intro} and the fact that $\HHH(T(n,\infty))$ is supported in even homological degrees (compare with \cite[Example 5.27]{Conners}). However, this is insufficient for proving \eqref{eq: interpolation intro}, so we instead reprove part (a) from scratch by using the computations of $\HY(T(n,mn))$ for all $m$ in \cite{GH}. We explicitly describe the graded algebra 
$$
\cA_n:=\bigoplus_{m=0}^{\infty}\HY(T(n,mn))
$$
and identify the connecting maps $\a_n$ from \eqref{eq: colimit intro 2} with multiplication by an explicit element $\Delta_n\in \cA_n$. This allows us to identify the colimit  \eqref{eq: colimit intro 2} with the localization $\cA_n[\Delta_n^{-1}]$ which is then computed by algebraic methods, see Theorem \ref{thm: y-ified projector}.

Next, we develop $y$-ification for $\fgl_N$ homology, denoted $\HY_{\fgl_N}(L)$. For $N=2$, it agrees with the Batson-Seed deformation of Khovanov homology \cite{BS}, and for general $N$, it agrees with the Cautis-Kamnitzer deformation \cite{CK}.
We develop the $y$-ification of the Rasmussen spectral sequence from $\HY(L)$ to $\HY_{\fgl_N}(L)$ for a link $L$, and show that it extends to $T(n,\infty)$.
The following is our next main result:

\begin{theorem}
\label{thm: y ified main intro}
For $N,n \ge 1$, there is a spectral sequence with the $E_1$ page $\HY(T(n,\infty))$ that abuts to $\HY_{\fgl_N}(T(n,\infty))$. The first differential
$d_N$ is defined by
$d_N(y_i)=d_N(u_k)=0$
and
\begin{equation}
\label{eq: dNy intro}
d_N(\xi_0+\xi_1z+\ldots+\xi_{n-1}z^{n-1})=
(u_0+u_1z+\ldots+u_{n-1}z^{n-1})^N\mod p(z).
\end{equation}
Here $z$ is a formal parameter only used to concisely express the differential, and $p(z)=\prod_{i=1}^{n}(z-y_i)$.
Furthermore, the spectral sequence collapses at the $E_2$ page. 
\end{theorem}

To unpack the equation \eqref{eq: dNy intro}, observe that $p(z)$ is a monic polynomial in $z$ with coefficients in $\k[y_1,\ldots,y_n]$. The remainder in long dividing 
$(u_0+u_1z+\ldots+u_{n-1}z^{n-1})^N$ by $p(z)$ is a degree $n-1$ polynomial in $z$ with coefficients in $\k[y_1,\ldots,y_n,u_0,\ldots,u_{n-1}]$. The coefficient at $z^k$ in the latter defines $d_{N}(\xi_k)$.  We give closed formulas for $d_{N}(\xi_k)$ in Lemma \ref{lem: d_N explicit} and prove the collapse of the spectral sequence in Theorem \ref{thm: y-ified differential}. 

\begin{remark}
If we substitute $z=y_i$ in \eqref{eq: dNy intro} and use \eqref{eq: interpolation intro}, we get $d_N(\theta_i)=x_i^N$ which indeed follows from the definition of the Rasmussen spectral sequence. 
\end{remark}

We deduce Theorem \ref{thm:page1theorem} from Theorem \ref{thm: y ified main intro} by specializing $y_i=0$ and observing that the Rasmussen spectral sequence behaves well with respect to such a specialization. At $y_i=0$ we get $p(z)=z^n$, hence
$$
d_N(\xi_0+\xi_1z+\ldots+\xi_{n-1}z^{n-1})=
(u_0+u_1z+\ldots+u_{n-1}z^{n-1})^N\mod z^n.
$$
which is indeed equivalent to \eqref{eq:dNdifferential}.

Our treatment of $y$-ification of $\fgl_N$ homology also allows us to define $y$-ified $S^n$-colored $\fgl_N$ homology; see section~\ref{sec: colored homology}.

\subsection{More general potentials}

For various applications, it is useful to consider deformations of $\fgl_N$ Khovanov--Rozansky homology where the homology of the unknot is given by $\Q[x]/(\partial W(x))$ for some polynomial $W(x)$. By making the coefficients of $W(X)$ indeterminates, one obtains equivariant Khovanov--Rozansky homology.

In this setting, we can describe both the $y$-ified $(d_{\partial W}^y)$ and the non $y$-ified $(d_{\partial W})$ variants of the first page of the Rasmussen spectral sequence as follows.
As in \eqref{eq: dNy intro}, define the following polynomials in a formal parameter $z$:
$$
u(z)=u_0+u_1z+\ldots+u_{n-1}z^{n-1},\
\xi(z)=\xi_0+\xi_1z+\ldots+\xi_{n-1}z^{n-1},\ p(z)=\prod_{i=1}^{n}(z-y_i).
$$
\begin{theorem}
\label{thm: dW intro}
We have
\begin{equation}
\label{eq: dW intro}
d_{\partial W}^y(\xi(z))=\partial W(u(z))\bmod p(z),\quad\ d_{\partial W}(\xi(z))=\partial W(u(z))\bmod z^n.
\end{equation}
\end{theorem}
As before, $d_{\partial W}$ can be obtained from $d_{\partial W}^y$ by specializing $y_i=0$.

\section*{Acknowledgments}

We are grateful to Luke Conners, Tom Mrowka, Alexei Oblomkov, Jacob Rasmussen, and Paul Wedrich for useful discussions. EG was partially supported by the NSF grant DMS-2302305. WB and JW were partially supported by the Simons Collaboration on New Structures in Low-Dimensional Topology, and JW was also partially supported by the NSF MSPRF grant DMS-2303401.

\section{Background}

\subsection{Gradings}
\label{ss:gradings}
The three gradings on triply graded Khovanov--Rozansky link homology are referred to as the $Q$-grading, the $T$-grading, and the $A$-grading.  The $Q$-grading is the quantum (i.e.~the grading internal to all rings and (bi)modules), the $T$-grading is the cohomological grading on complexes, and the $A$-grading is the Hochschild cohomological grading.

If $V=\bigoplus_{i,j,k\in \Z}V^{i,j,k}$ is a $\Z^3$-graded vector space, then the (tri)degree of a homogeneous element $v\in V^{i,j,k}$ may be written $\deg_{Q,T,A}(v):=(\deg_Q(v),\deg_T(v),\deg_A(v))=(i,j,k)$.  We may also express the degree of $v$ exponentially (using the term \emph{weight}), writing $\wt(v)=Q^i T^j A^k$.

We also use the symbol $Q^i T^j A^k$ for the corresponding grading shift functor:
\[
(Q^i T^j A^k V)^{i',j',k'} = V^{i'-i,j'-j,k'-k}
\]

\begin{remark}
When clarification is desired, we may write the various grading groups with subscripts indicating the relevant degrees.  For instance, we may refer to $\Z_Q$-graded modules, or $\Z_Q\times \Z_T$-graded complexes, and so on.
\end{remark}

\begin{remark}
By default the notation $\Hom(-.-)$ will always mean enriched homs in the appropriate sense.  For instance, if $R$ is a $\Z_Q$-graded ring, and $M,N$ are graded $R$-modules, then $\Hom_R(M,N)$ will always mean the $\Z_Q$-graded abelian group of (not necessarily degree preserving) $R$-linear homs. 

Similarly, if $X$ and $Y$ are complexes of $\Z_Q$-graded $R$-modules, then $\Hom_R(X,Y)$ will mean the $\Z_Q\times \Z_T$-graded complex of homs.
\end{remark}

As is common in the link homology literature (e.g., \cite{EH}) we will sometimes change variables to
\begin{equation}
\label{eq: qta to QTA}
q=Q^2,\ t=Q^{-2}T^2,\ a=Q^{-2}A.
\end{equation}
In other words, $(\deg_Q,\deg_T,\deg_A)=(2\deg_q-2\deg_t-2\deg_a, 2\deg_t, \deg_a)$.

Throughout this paper, we will be working with various graded (super)polynomial rings in variables $x_i$, $y_i$, or $\theta_i$; by convention the weights of these variables are
\begin{equation}\label{eq:weights of xyt}
\wt(x_i)=q \, , \qquad \wt(y_i)=t \, , \qquad \wt(\theta_i)=a.
\end{equation}
where $q,t,a$ are as in \eqref{eq: qta to QTA}.  Written additively, the (tri)degrees are
\begin{equation}\label{eq:degrees of xyt}
\deg_{Q,T,A}(x_i)=(2,0,0) \, , \qquad \deg_{Q,T,A}(y_i)=(-2,2,0) \, , \qquad \deg_{Q,T,A}(\theta_i)=(-2,0,1).
\end{equation}
The variables $\theta_i$ are odd (meaning they square to zero and anticommute amongst themselves) since they carry Hochschild cohomological degree 1.  We will also abbreviate by writing alphabets in these various variables as in
\begin{equation}\label{eq:alphabet soup}
\xx=(x_1,\ldots,x_n) \, ,\qquad \yy=(y_1,\ldots,y_n) \, , \qquad \ttheta=(\theta_1,\ldots,\theta_n),
\end{equation}
where the integer $n$ will be understood from context.  For instance, the notation $\k[\xx,\yy,\ttheta]$ denotes the $\Z_Q\times \Z_T\times \Z_A$-graded super polynomial ring generated by even variables $x_i,y_i$, and odd variables $\theta_i$.  

\subsection{Rouquier complexes}
\label{sec: rouquier}

Fix an integer $n\geq 1$ and a commutative ring $\k$ (soon we will specialize $\k$ to a field of characteristic zero). Let $R:=\k[\xx]$ be the $\Z_Q$-graded polynomial ring, which we regard as a $\Z_Q$-graded algebra by declaring that $x_i$ has degree 2, i.e.~$\wt(x_i)=Q^2=q$ in the above notation. Let $R\text{-gmod-}R$ denote the category of $\Z_Q$-graded $(R,R)$-bimodules, with morphisms given by (not necessarily degree zero) bimodule maps.

Introduce a second alphabet of $x$-variables, denoted $\xx'=(x_1',\ldots,x_n')$.  We will identify $R\otimes R=\k[\xx,\xx']$ with $x_i\otimes 1$ identified with $x_i$, and $1\otimes x_i$ identified with $x_i'$.  This gives an identification of categories
\begin{equation}\label{eq:bimod to mod}
\k[\xx]\text{-gmod-}\k[\xx]=\k[\xx,\xx']\text{-gmod}.
\end{equation}
On the left hand side of this identification we have the monoidal structure $\otimes_R$ which we denote by $\star$.  On the right hand side this is most conveniently written as a bilinear functor
\[
(\k[\xx,\xx']\text{-gmod})\times (\k[\xx',\xx'']\text{-gmod})\to \k[\xx,\xx'']\text{-gmod} \, \qquad (M,N)\mapsto M\otimes_{\k[\xx']}N.
\]
The unit object for this tensor product is $R$ which we will sometimes denote as $\one_n$.

Let $\Ch(R\text{-gmod-}R)$ denote the category of complexes of graded bimodules, with grading-preserving differentials.  An object $X$ of this category is $\Z_Q\times \Z_T$-graded with differentials of weight $\wt(d_X)=T$.  The monoidal structure $\star$ extends to  $\Ch(R\text{-gmod-}R)$.

We recall some basic definitions for link homology. 
Let $$B_i=\k[\xx,\xx'] \Big/ (x_i+x_{i+1}=x'_{i}+x'_{i+1},x_ix_{i+1}=x'_{i}x'_{i+1},x_j=x'_j \ \text{for}\ j\neq i)$$

We interpret $B_i$ as a graded $(R,R)$-bimodule, via the identification \eqref{eq:bimod to mod}.  The category of Soergel bimodules $\SBim_n$ is defined as the smallest full subcategory of the category of graded $(R,R)$-bimodules containing $B_i$ and $R$ and closed under tensor product $\star$, direct sums, grading shifts and direct summands.

There are canonical maps $b_i\in \Hom(B_i, R)$ and $b^*_i\in \Hom(R,B_i)$ with $\wt(b_i)=Q^0$, $\wt(b_i^*)=Q^2$, defined by
$$
b_i(1)=1,\ b_i^*(1)=x_i-x'_{i+1}.
$$
We define the Rouquier complexes by 
$$
F_i=\left[ 
\begin{tikzcd}
Q^{-1}\underline{B_i}\arrow {r}{b_i} & Q^{-1}R 
\end{tikzcd}\right],\ 
F^{-1}_i=\left[ 
\begin{tikzcd}
QR\arrow {r}{b^*_i} & Q^{-1}\underline{B_i} 
\end{tikzcd}\right].
$$
The terms in homological degree zero are underlined.
\begin{theorem}\cite{Rouquier}
The complexes $F_i$ and $F_i^{-1}$ satisfy braid relations 
$$
F_i\star F_{i+1}\star F_i\simeq F_{i+1}\star F_i\star F_{i+1},\ F_i\star F_j\simeq F_j\star F_i\ (|i-j|>1),\ F_{i}\star F_i^{-1}\simeq R$$
up to homotopy.
\end{theorem}
 By tensoring these, for any braid $\beta$ one can construct a  complex $F(\beta)$   which is well-defined up to homotopy equivalence. We denote the differential in this complex by $d$.

 \begin{example}
 \label{ex: rightmost R}
 If $\beta$ is a positive braid, then $F(\beta)$ has a unique copy of $R$ in the rightmost homological degree shifted by $Q^{-\ell(\beta)}T^{\ell(\beta)}$. It defines a special chain map 
 \begin{equation}
 \label{eq: rightmost R}
 Q^{-\ell(\beta)}T^{\ell(\beta)}R\to F(\beta).
 \end{equation}
 \end{example}

\begin{lemma}\cite{GH}
\label{lem: dot sliding}
Let $\b\in \Br_n$ be a braid with underlying permutation $w\in S_n$.  Then for all $i$ there exists $\eta_i\in \mathrm{End}^{2,-1}(F(\beta))$ such that 
\begin{align}
\label{eq: def xi}
[d,\eta_i]&=x_{i}-x'_{w\inv(i)}\\
[\eta_i,\eta_j]&=0\\
\eta_i^2&=0
\end{align}
\end{lemma}
We remind the reader that the notation $\eta_i\in \mathrm{End}^{2,-1}(F(\beta))$ means that $\wt(\eta_i)=Q^2T\inv$.

We will refer to $\eta_i$ as to {\em dot-sliding homotopies.}

\subsection{$y$-fication}

\begin{definition}
Let $V$ be a $\Z_A\times \Z_Q\times \Z_T$ graded (resp. $\Z_Q\times \Z_T$ graded)  $\k$-module. We define the completed tensor product $V\cotimes\k[\yy]=V\cotimes\k[y_1,\ldots,y_n]$ as follows.
A homogeneous element of weight $A^{k}Q^{m}T^{s}$ in $V\cotimes\k[\yy]$  is a (possibly infinite) linear combination 
\begin{equation}
\label{eq: completed tensor}
\sum_{k_1,\ldots,k_n\geq 0} f_{k_1,\ldots,k_n} y_1^{k_1}\cdots y_n^{k_n}
\end{equation}
where $f_{k_1,\ldots,k_n}$ is a homogeneous element of $V$ of weight $A^kQ^{m+2K}T^{s-2K}$, $K=k_1+\ldots+k_n$. 
An arbitrary element of $V\cotimes\k[\yy]$ is a finite linear combination of homogeneous elements.
\end{definition}

\begin{remark}
It is easy to see that 
$$
V[y_1,\ldots,y_n]\subset V\cotimes\k[y_1,\ldots,y_n]\subset V\llbracket y_1,\ldots,y_n\rrbracket.
$$
The second inclusion is strict if $V\neq 0.$
\end{remark}

Recall the conventions from Equations \eqref{eq:alphabet soup}, \eqref{eq:weights of xyt}.  If $X,Y\in \Ch(\SBim_n)$ are   complexes of Soergel bimodues, then we will adjoin variables $y_1,\ldots,y_n$ to the hom complex $\Hom(X,Y)$, obtaining
\begin{equation}\label{eq:complex with ys}
\Hom(X,Y)\cotimes\k[\mathbf{y}].
\end{equation} 
There is a differential on \eqref{eq:complex with ys}, given by
\[
d\left(\sum_{k_1,\ldots,k_n\geq 0} f_{k_1,\ldots,k_n} y_1^{k_1}\cdots y_n^{k_n}\right) := \sum_{k_1,\ldots,k_n\geq 0} d(f_{k_1,\ldots,k_n}) y_1^{k_1}\cdots y_n^{k_n}.
\]
When $X=Y$ also abbreviate by writing $\End(X)\cotimes\k[\mathbf{y}]:=\Hom(X,X)\cotimes\k[\mathbf{y}]$.  We say that $f\in \Hom(X,Y)\cotimes\k[\mathbf{y}]$ is \emph{polynomial} if the terms $f_{k_1,\ldots,k_n}$ are zero for all but finitely many tuples $(k_1,\ldots,k_n)$.  The polynomial elements form a subcomplex, denoted by
\[
\Hom(X,Y)[\mathbf{y}]\subset \Hom(X,Y)\cotimes\k[\mathbf{y}]
\]
\begin{remark}
\label{rem: poly 1}
If $X$ and $Y$ are such that $\Hom^{i+2K,j-2K}(X,Y)=0$ for $K\gg 0$, then we have $\Hom(X,Y)\cotimes\k[\mathbf{y}] = \Hom(X,Y)\otimes \k[\mathbf{y}]$.  In particular, this is the case if $X$ and $Y$ are bounded complexes.
\end{remark}

\begin{definition}\label{def:yified cat}
Let $w\in S_n$ be a given permutation.  We define a category $\mathcal{Y}_{n,w}$ as follows.  Objects of this category are formal expressions $\tw_{\alpha}(X)$ where $X\in \Ch(\SBim_n)$ and $\alpha\in \End(X)\cotimes\k[\mathbf{y}]$ is an element of weight $T$ such that
\begin{equation}\label{eq:y curvature}
(d_X+\alpha)^2 = \sum_{i=1}^n(x_{i}-x_{w\inv(i)}')y_i
\end{equation}
The expression $D:=d_X+\alpha$ is called the \emph{total differential} acting on $\tw_{\alpha}(X)$.  The homs in $\mathcal{Y}_{n,w}$ are the $\Z_Q\times \Z_T$-graded complexes given by
\[
\Hom_{\mathcal{Y}_{n,w}}(\tw_{\alpha'}(X'),\tw_\alpha(X)) = \Hom(X,Y)\cotimes\k[\mathbf{y}]
\]
with differential given by the supercommutator with $D$, i.e.
~$f\mapsto [D,f]$.
\end{definition}

Up to isomorphism, the object $\tw_\alpha(X)$ only depends on the total differential $d_X+\alpha$. Consequently, we have
\[
\tw_{\alpha}(X) \cong \tw_{\alpha_+}(X,d_X+\alpha_0),
\]
where we have written $\alpha_0:=\alpha|_{y_1=\cdots=y_n=0}$ and $\alpha_+:=\alpha-\alpha_0$.  Thus, we can always write objects of $\mathcal{Y}_{n,w}$ as expressions $\tw_\alpha(X)$ with $\alpha_0=0$. In this case we think of $\tw_\alpha(X)$ as $X$ with a deformed differential.

\begin{remark}
Given $X\in \Ch(\SBim_n)$ we may write $X^y$ for an object $\tw_\alpha(X)\in \mathcal{Y}_{n,w}$ with $\alpha_0=0$ (as in the preceding discussion).  In practice, we only write $X^y$ when the choice of deformation is unique up to isomorphism.  
\end{remark}

For establishing the existence and uniqueness of lifts we recall some results from \cite[Section 2]{GH}.

\begin{lemma}\label{lemma:obstructions}
Suppose we are given a permutation $w\in S_n$ and a complex $X\in \Ch(\SBim_n)$ on which the action of $x_i-x_{w\inv(i)}'$ is null-homotopic for all $1\leq i\leq n$.  Suppose $X$ satisfies the ext-vanishing condition that any closed element of $\End(X)$ of weight $Q^{2+2K}T^{-2K}$ (with $K\in \Z_{\geq 1}$) is null-homotopic.  Then there exists a lift $X^y\in \mathcal{Y}_{n,w}$. 
\end{lemma}

\begin{proof}
For the reader's convenience, we sketch the proof and refer to \cite{GH} for more details. We write
$$
\alpha=\sum_{k=1}^{\infty}\alpha_k
\ , \quad \text{where}\quad \ \alpha_k=\sum_{k_1+\ldots+k_n=k}\alpha_{k_1,\ldots,k_n}y_1^{k_1}\cdots y_n^{k_n}.
$$
Since $\alpha$ has weight $T$, all coefficients $\alpha_{k_1,\ldots,k_n}$ of $\alpha_k$ have weight $Q^{2k}T^{1-2k}$.
The equation  \eqref{eq:y curvature} unpacks to an infinite sequence of equations:
\begin{align*}
[d,\alpha_1]&=\sum_{i=1}^n(x_i-x_{w\inv(i)}')y_i\\
[d,\alpha_k]+\sum_{i+j=k}\alpha_i\alpha_j&=0,\ k>1.
\end{align*}
For the first equation, we can use null-homotopies for $x_i-x_{w\inv(i)}'$ as coefficients of $\alpha_1$. For $k>1$, we construct  $\alpha_k$ by induction, assume that we have constructed $\alpha_1,\ldots,\alpha_{k-1}$. One can check that the sum $\sum_{i+j=k}\alpha_i\alpha_j$ is closed under $d$, and to construct $\alpha_k$ it is sufficient to prove that it is null-homotopic.

Observe that for $i+j=k$ all coefficients in $\alpha_i\alpha_j$ have weights 
$$
(Q^{2i}T^{1-2i})\cdot (Q^{2j}T^{1-2j})=Q^{2k}T^{2-2k}=Q^{2+2(k-1)}T^{-2(k-1)}. 
$$
By our assumption, this implies that $\sum_{i+j=k}\alpha_i\alpha_j$ is null-homotopic, and the result follows.
\end{proof}

For the uniqueness of lifts of $X$ the following is useful.

\begin{lemma}\label{lemma:lifting maps}
Suppose we are given complexes $X,Y\in \Ch(\SBim_n)$ and lifts $X^y,Y^y\in \mathcal{Y}_{n,w}$, and let $f\colon X\to Y$ be a degree zero chain map.  Suppose that we have the ext-vanishing condition that any closed element of $\Hom(X,Y)$ of weight $Q^{2K}T^{1-2K}$ (with $K\in \Z_{\geq 1}$) is null-homotopic. Then there exists a lift of $f$ to a degree zero chain map $f^y\colon X^y\to Y^y$.
\end{lemma}
One can also formulate a statement on the uniqueness of $f^y$ up to homotopy, but we will not need it. 

\begin{lemma}\label{lemma:lift of equiv}
If $X\in \Ch(\SBim_n)$ is contractible, then so is $X^y$ for any lift $X^y\in \mathcal{Y}_{n,w}$.  If $\phi\colon X\to X'$ is a homotopy equivalence then any lift $\phi^y\colon X^y\to (X')^y$ is a homotopy equivalence in $\mathcal{Y}_{n,w}$.
\end{lemma}

Combining Lemmas \ref{lemma:lifting maps} and \ref{lemma:lift of equiv} gives the following.

\begin{lemma}\label{lemma:uniqueness of lifts}
Let $X\in \Ch(\SBim_n)$ be given, and assume that we have the ext-vanishing condition that any closed element of $\End(X)$ of weight $Q^{2K}T^{1-2K}$ (with $K\in \Z_{\geq 1}$) is null-homotopic. Then any $X^y\in \mathcal{Y}_{n,w}$ lifting $X$ is unique up to homotopy equivalence.
\end{lemma}
\begin{proof}
Let $\tw_{\alpha}(X)$ and $\tw_{\alpha'}(X)$ be two lifts of $X$ to $\mathcal{Y}_{n,w}$. The ext-vanishing condition tells us that $\id_X$ has a lift to a closed morphism  $\tw_{\alpha}(X)\to \tw_{\alpha'}(X)$, by Lemma \ref{lemma:lifting maps}. This lift is a homotopy equivalence by Lemma \ref{lemma:lift of equiv}.
\end{proof}

\begin{definition}\label{def:yified Rouquier}
Given a braid $\b\in \Br_n$ with underlying permutation $w\in S_n$, we let $F^y(\b)\in \mathcal{Y}_{n,w}$ be defined by
\[
F^y(\b):=\tw_{\sum_i \eta_i y_i}(F(\b))
\]
where the $\eta_i$ are as in Lemma \ref{lem: dot sliding}. 
\end{definition}

Note that by Lemma \ref{lem: dot sliding} we have
$$
\left(d+\sum_i \eta_i y_i\right)^2=d^2+\sum_{i}[d,\eta_i]y_i+\left(\sum_i \eta_i y_i\right)^2=\sum_{i=1}^n(x_{i}-x_{w\inv(i)}')y_i.
$$
The last term vanishes since $\eta_i$ square to zero and anticommute.

\begin{example}
For the elementary braid $\sigma_i,\sigma_i\inv$  the corresponding deformed Rouquier complex is visualized as in the diagrams
$$
F^y(\sigma_i)=\left[ 
\begin{tikzcd}
Q^{-1}\underline{\yBi}\arrow[bend left]{r}{b_i} & Q^{-1}\yR \arrow[bend left,pos=0.55]{l}{b^*_i(y_i-y_{i+1})}
\end{tikzcd}\right],\ 
F^y(\sigma_i\inv)=\left[ 
\begin{tikzcd}
Q\yR\arrow[bend left,pos=0.55]{r}{b^*_i} & Q^{-1}\underline{\yBi} \arrow[bend left,pos=0.4]{l}{b_i(y_i-y_{i+1})}
\end{tikzcd}\right].
$$
The indicated arrows define the total differential $D$.  Note that the cohomological degree of the $y_i$ is 2, which gives the leftward arrows cohomological degree $+1$, rather than the degree $-1$ which they would otherwise have.
\end{example}

There is a tensor product operation $\mathcal{Y}_{n,w}\star \mathcal{Y}_{n,v}\to \mathcal{Y}_{n,wv}$, defined as follows:
\[
\tw_{\alpha}(X)\star \tw_{\alpha'}(X') := \tw_{\gamma}(X\star X'),
\]
where $\gamma=\alpha\star \id_{X'} + \id_X\star \alpha'|_{y_i\mapsto y_{w(i)}}$.

The properties of $y$-ified Rouquier complexes $F^y(\beta)$ can be summarized as follows.

\begin{theorem}\cite{GH}
a) Each Rouquier complex admits a unique-up-to-homotopy deformation to an object of the appropriate category $\mathcal{Y}_{n,w}$.

b) For any braids $\b,\b'\in \Br_n$ we have $F^y(\b)\star F^y(\b')\simeq F^y(\b\b')$.  In particular the complexes $F^y(\sigma_i^\pm)$ satisfy the braid relations up to homotopy.
\end{theorem}

We give a proof for the reader's convenience.

\begin{proof}
For part (a), observe that $F(\beta)$ is invertible hence $\End(F(\beta))=\End(R)=R$, and all closed endomorphisms of nonzero $T$-degree are null-homotopic. Now the result follows from Lemmas \ref{lemma:obstructions} and \ref{lemma:uniqueness of lifts}. Part (b) follows from (a). 
\end{proof}

\begin{definition}
For each reflection $r=(i\ j)$ in $S_n$ we define the Koszul $y$-ification
$$
K_r=K_{ij}=\left[
\begin{tikzcd}
Q^{-1}R\arrow[bend left, pos=0.4]{r}{x_i-x_j} & Q\underline{R} \arrow[bend left]{l}{y_i-y_j}
\end{tikzcd}
\right]\in \mathcal{Y}_{n,r}.
$$
\end{definition}

The following facts were proved in \cite{GH,GHM}.

\begin{lemma}(\cite[Proposition 4.3]{GH})
\label{lem: skein}
a) There is a (unique up to homotopy) degree zero chain map $\psi_i:\yTi\to \yTii$ such that $\Cone(\psi_i)\simeq K_{i,i+1}$. As a consequence, we have an exact triangle in $\mathcal{Y}_{n,s_i}$:
$$
\yTi\xrightarrow{\psi_i} \yTii\to K_{i,i+1}\to \yTi[1]
$$

b) There is a (unique up to homotopy) chain map $\a_i:Q^{-2}T^2\yTii\to \yTi$ such that $\Cone(\a_i)$ is a complex build of two copies of $B_i$. 

c) We have $\a_i\circ \psi_i=\psi_i\circ \a_i=y_i-y_{i+1}$.
\end{lemma}

Explicitly, the maps $\psi_i$ and $\a_i$ are given by the following diagram:
$$
\psi_i:
\begin{tikzcd}
Q^{-1}\underline{\yBi}\arrow[bend left]{rr}{b_i} \arrow[swap,pos=0.2]{ddrr}{1} & & Q^{-1}\yR \arrow{ll}{b^*_i(y_i-y_{i+1})} \arrow[pos=0.2]{ddll}{y_i-y_{i+1}}\\ 
 & & \\
Q\yR\arrow{rr}{b^*_i} & & Q^{-1}\underline{\yBi} \arrow[bend left,pos=0.45]{ll}{b_i(y_i-y_{i+1})}
\end{tikzcd}
\quad
\a_i:
\begin{tikzcd}
Q^{-1}\underline{\yBi}\arrow[bend left]{rr}{b_i}  & & Q^{-1}\yR \arrow{ll}{b^*_i(y_i-y_{i+1})} \\ 
 & & \\
Q\yR\arrow{rr}{b^*_i} \arrow[pos=0.2]{uurr}{1}& & Q^{-1}\underline{\yBi} \arrow[bend left,pos=0.45]{ll}{b_i(y_i-y_{i+1})} \arrow[swap,pos=0.2]{uull}{y_i-y_{i+1}}
\end{tikzcd}
$$

\begin{definition}
Let $A_{ij}=\sigma_{j-1}\cdots\sigma_{i+1}\sigma_i^2\sigma_{i+1}\inv\cdots \sigma_{j-1}\inv$ denote Artin's generator of the pure braid group. 
 \[
 A_{ij} \ =\  
 \begin{minipage}{1.5in}
\labellist
\small
\pinlabel $i$ at 15 -7
\pinlabel $j$ at 68 -7
\endlabellist
\includegraphics[scale=.8]{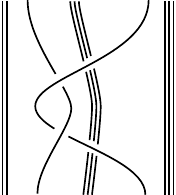}
\end{minipage}
 \]
 \vskip7pt
 \end{definition}

\begin{lemma}
\label{lem: rightmost R y-ified}
Let $\beta=A_{i_1,j_1}\cdots A_{i_k,j_k}$ be a positive pure braid presented as a product of $A_{ij}$. Then there exists morphisms in $\mathcal{Y}_{n,1}$:
$$
\psi_{\beta}: F^y(\beta)\to R,\ \a_{\beta}: Q^{-\ell(\beta)}T^{\ell(\beta)}R\to F^y(\beta)
$$
such that 
$$
\psi_{\beta}\circ \rho_{\beta}=\rho_{\beta}\circ \psi_{\beta}=(y_{i_1}-y_{j_1})\cdots (y_{i_k}-y_{j_k}).
$$
Furthermore, $\a_{\beta}$ deforms the map from \eqref{eq: rightmost R}.
\end{lemma}

\begin{proof}
First, by Lemma \ref{lem: skein} we can construct maps 
$$\psi_{i}: F^y(\sigma_i^2)\to R,\ \a_{i}: Q^{-2}T^{2}R\to F^y(\sigma_i^2)
$$
with the desired properties. Since $A_{ij}$ is conjugate to $\sigma_i^2$, we get 
$$
\Hom_{\mathcal{Y}_{n,1}}\left(F^y(A_{ij}),R\right)=\Hom_{\mathcal{Y}_{n,1}}\left(F^y(\sigma_i^2),R\right),
$$
$$
\Hom_{\mathcal{Y}_{n,1}}\left(R,F^y(A_{ij})\right)=\Hom_{\mathcal{Y}_{n,1}}\left(R,F^y(\sigma_i^2)\right),
$$
and this yields maps
$$\psi_{ij}: F^y(A_{ij})\to R,\ \a_{ij}: Q^{-2}T^{2}R\to F^y(A_{ij})
$$
which compose to $y_i-y_j$. The result follows by multiplying $\psi_{ij}$ and $\a_{ij}.$
 \end{proof}

 \begin{remark}
 The maps $\psi_{\beta}$ and $\rho_{\beta}$ defined above might depend on the choice of factorization of $\beta$ as the product of Artin generators. We plan to study this dependence in the future work, but for the purposes of this paper we simply choose one such factorization. See also Lemma \ref{lem: connecting map Delta} for the case when $\beta$ is the full twist braid.
 \end{remark}

\section{Rasmussen spectral sequences}

In this section we define a family of spectral sequences relating triply graded Khovanov--Rozansky homology to $\mathfrak{gl}(N)$ homology. Our description is slightly different from the original construction of Rasmussen \cite{RasDiff} but equivalent to it, see Remark \ref{rmk: Rasmussen}.  

\subsection{Hoschschild (co)homology}
\label{ss:HH}

We start by reviewing Hochschild homology.   
\begin{definition}\label{def:Delta}
Define the chain complex $\Delta_n(\xx,\xx')$ to have underlying module $\k[\xx,\xx',\ttheta]$  (with conventions as in \eqref{eq:alphabet soup}), and differential
\begin{equation*}
    d_{\HH} = \sum_{i = 1}^n (x_i-x_i') \theta_i
\end{equation*}
We regard $\Delta_n(\xx,\xx')$ as a $\Z_Q\times \Z_A$-graded complex, with gradings as in \eqref{eq:weights of xyt}, the differential $d_{\HH}$ has weight $A=qa$.
\end{definition}

\begin{remark}\label{rmk:Delta_n}
The top monomial in $\Delta_n(\xx,\xx')$ has $\wt(\theta_1\cdots \theta_n)=a^n$.  The dual complex $$\Hom_{\k[\xx,\xx']}\left(\Delta_n(\xx,\xx'),\k[\xx,\xx']\right)$$ is isomorphic to $a^{-n}\Delta_n(\xx,\xx')$.  In turn, this is  the Koszul complex resolving $R=\k[\xx,\xx']/(x_i-x'_i)$ over $\k[\xx,\xx']$.
\end{remark}

\begin{definition}
\label{def: HH}
Let $M\in R\text{-gmod-}R$ be a graded bimodule.  Then we define $\HH(M)$ to be the homology of the complex $\left(M \otimes_{\k[\xx,\xx']} \Delta_n(\xx,\xx'),d_{\HH}\right)$, where we identify $R\otimes R=\k[\xx,\xx']$ as in \S \ref{sec: rouquier}. 
\end{definition}

We will regard $\HH(M)$ as $\Z_Q\times \Z_A$-graded $\k$-module.

\begin{remark}
In light of Remark \ref{rmk:Delta_n}, we have that $\HH(M)$ is the Hochschild cohomology of $M$ or, equivalently, the Hochschild homology of $M$ up to a shift:
$$
\HH(M)=a^n\Tor^{\k[\xx,\xx']}(M,R) = \Ext_{\k[\xx,\xx']}(R,M).
$$

See also \cite[Lemma 3.4]{GHMN}.
\end{remark}

\begin{remark}
\label{rem: HH is a functor}
It follows from the definition that $\HH$ is a functor, that is, for any bimodule morphism $M\to N$ we get a $\k$-linear map $\HH(M)\to \HH(N)$.
\end{remark}

The general properties of the $\Tor$ functor immediately imply the following:

\begin{proposition}
\label{prop: Tor invariance}
Let $F^{\bullet}$ be a free resolution of $M$ over $\k[\xx,\xx']$. Then the homology of $F^{\bullet}\otimes_{\k[\xx,\xx']}R$ is isomorphic to $a^{-n}\HH(M)$.
\end{proposition}

Below (in particular, to compare our constructions with \cite{RasDiff}) we will need some explicit models for such resolutions. 

\begin{lemma}
\label{lem: examples resolutions}
a) Let $M$ be a Soergel bimodule. Then $M\otimes_{\k[\xx']}\Delta_n(\xx',\xx'')$ is free a resolution of $M$.

b) Let $M=B_{i_1}\star \cdots \star B_{i_r}$, and let $F_{i}^{\bullet}$ be a free resolution of $B_i$ over $\k[\xx,\xx']$. Then $F_{i_1}^{\bullet}\star \cdots \star F_{i_r}^{\bullet}$ is a free resolution of $M$.
\end{lemma}

\begin{proof}
It is well known (e.g., \cite[Corollary 3.6]{GHMN}) that any Soergel bimodule is free as a left or as a right $R$-module. Therefore for any Soergel bimodules $M$ and $N$ and their free resolutions $F_M^{\bullet}$ and $F_N^{\bullet}$ the tensor product $F_M^{\bullet}\star F_N^{\bullet}$ is a free a resolution of $M\star N$. In other words, the derived tensor product of $M$ and $N$ agrees with the underived one. This implies (b).

For (a), observe that $M$ is free over $\k[\xx]$, so 
$M\otimes_{\k[\xx']}\Delta_n(\xx',\xx'')$ is free over $\k[\xx,\xx'']$.
\end{proof}

Next, we discuss well-known multiplicative structure on $\HH$.

\begin{proposition}\label{prop:HH monoidal}
The functor $\HH$ has is lax monoidal, in the sense that we have a canonical multiplication map
\begin{equation}
\label{eq: multiplication HH}
\HH(M)\otimes \HH(N)\to \HH(M\star N).
\end{equation}
for bimodules $M,N\in R\text{-mod-}R$, satisfying appropriate naturality, associativity, and unit constraints.
\end{proposition}

\begin{proof}
This is well known, but we include the details for completeness.  We introduce sets of variables $\xx,\xx',\xx'',\xx'''$, so that each of $\k[\xx], \k[\xx'], \k[\xx''], \k[\xx''']$ may be regarded as a copy of $R$.  We regard $M$ as a module over $\k[\xx,\xx']$ and $N$ as a module over $\k[\xx'',\xx''']$.  For convenience we regard $R$ as a module over $\k[\xx',\xx'']$ (in which $x_i'-x_i''$ acts by zero for all $i$).  Thus, the tensor product $M\star N$ may be written as $M\otimes_{\k[\xx']} R\otimes_{\k[\xx'']} N$.

Now, the two sides of \eqref{eq: multiplication HH} may be computed as
\begin{align*}
    \HH(M) \otimes \HH(N) &= \left(M \otimes_{\k[\xx,\xx']} \Delta_n(\xx,\xx')\right) \otimes_\Q \left(N \otimes_{\k[\xx'',\xx''']} \Delta_n(\xx'',\xx''') \right) \\
    & \cong \left(M \otimes_\k N\right) \otimes_{\k[\xx,\xx',\xx'',\xx''']} \left( \Delta_n(\xx,\xx') \otimes_\Q \Delta_n(\xx'',\xx''')\right),
\end{align*}
where the isomorphism simply switches factors in the tensor product, and
\begin{equation*}
    \HH(M \otimes_{\k[\xx']} N) = (M \otimes_{\k[\xx']} R\otimes_{\k[\xx'']} N) \otimes_{\k[\xx,\xx''']} \Delta_n(\xx,\xx''').
\end{equation*}
The multiplication map is then a composite of two parts: first, the quotient identifying $x_i'$ with $x_i''$, and second the map
\begin{equation*}
    \left( \Delta_n(\xx,\xx') \otimes_\k \Delta_n(\xx'',\xx''') \right) / (x_i'-x_i'') \to \Delta_n(\xx,\xx''')
\end{equation*}
given by multiplication on the exterior algebras underlying the $\Delta_n$. This multiplication gives a chain map since, modulo $x_i'-x_i''$, $$(x_i - x_i') + (x_i'' - x_i''') = x_i-x_i'''.$$
Clearly the multiplication map just constructed is natural in the arguments $M,N$, since it involves only terms acting on the $\Delta_n$ so commutes with any map on the other tensor factors $M$ and $N$.
\end{proof}

Now, we forget its cohomological origins, and regard $\HH$ as a functor from $\SBim_n$ to $\Z_Q\times \Z_A$-graded $\k$-modules.  
\begin{definition}\label{def:HHH}
If $X\in \Ch(\SBim_n)$ then we let $\CHH(X)$ denote the $\Z_Q\times \Z_T\times \Z_A$-graded complex obtained by applying $\HH$ termwise to $X$.  Let $\HHH(X)$ denote the homology of $\CHH(X)$. 

For a braid $\b\in \Br_n$ we will abbreviate by writing $\CHH(\b):=\CHH(F(\b))$, and similarly for $\HHH(\b)$. 
\end{definition}

\begin{theorem}[\cite{KhSoergel}]
The homology $\HHH(\beta)$ is a topological invariant of the link obtained as the closure of $\beta$, up to an overall shift in tridegree.
\end{theorem}

\begin{remark}
We recall how to normalize $\HHH(\b)$ to obtain a well-defined link invariant.  Given $\b\in \Br_n$, let $e=e(\b)$ be the number of positive crossings minus the number of negative crossings, and let $c=c(\b)$ denote the number of components of the link represented by $\b$.  We define the \emph{normalizing exponent} of the braid $\b$ by the formula
\begin{equation}\label{eq:normalizing exponent}
\delta(\b):=\frac{1}{2}(e(\b)+c(\b)-n(\b)).
\end{equation} 
The normalized Khovanov--Rozansky complex of $\b$ is
\[
(a t^{-1/2}q^{-1/2})^{\delta(\b)} \Big(t^{(c-n)/2} \CHH(F(\b))\Big)
\]
\end{remark}

\begin{remark}
Sometimes we will need to emphasize the role of the coefficient ring $\k$ below. In this case, we will write $\HHH(\beta;\k)$ instead of $\HHH(\beta)$.
\end{remark}

\begin{example}
Recall that $O_n$ denotes the $n$-component unlink, which is the closure of the $n$-strand identity braid.  The associated Rouquier complex is $R$, and (using notation from \S \ref{ss:gradings}) we get
$$
\HHH(O_n)=\HH(R)=\k[\xx,\ttheta]
$$
\end{example}

The lax monoidal structure on the functor $\HH$ from Proposition \ref{prop:HH monoidal} immediately gives us a (degree zero, closed) multiplication map
\begin{equation}
\label{eq: multiplication HHH}
\CHH(X)\otimes \CHH(X')\to \CHH(X\star X').
\end{equation}

When $X=R$ and $X'=F^y(\b)$ we obtain the following corollary.

\begin{corollary}
There is an action of $\HHH(O_n)=\k[\xx,\ttheta]$ on $\HHH(\beta)$ for any $\beta$. If $w_{\beta}$ is the corresponding permutation then the action factors through
$$
\k[\xx,\ttheta]/(x_i-x_{w_{\beta}(i)},\theta_i-\theta_{w_{\beta}(i)}).
$$
\end{corollary}

\begin{proof}
The only nontrivial part is the relations for $x_i$ and $\theta_i$ which follow from \ref{lem: dot sliding} and \cite[Section 3.4]{GH}.
\end{proof}

\subsection{Rasmussen differentials}
\label{ss: Rasmussen}

In this section we recall the family of differentials $d_N$ on Hochschild homology of Soergel bimodules, constructed in \cite{RasDiff}. 
From now on we assume that $\k$ is a field of characteristic zero.

Recall that $p_{N+1}(\xx)=\sum_{i=1}^{n}x_i^{N+1}$ is the power sum in $x_i$.

\begin{lemma}\label{lemma:res and h}
Let $F^{\bullet}$ be a free resolution of $M\in \SBim_n$.  Then there exists $d_N\in \End(F^{\bullet})$ of weight $q^Na\inv$ such that $[d_F,d_N]=\frac{1}{N+1}(p_{N+1}(\xx)-p_{N+1}(\xx'))$. Moreover, $d_N$ is unique up to homotopy in the sense that if $(F'^{\bullet},d_N')$ is such pair and $\phi\colon F^{\bullet}\to F'^{\bullet}$ is a lift of $\id_M$, then $d_N'\circ \phi-\phi\circ d_N$ is closed and exact.  Furthermore, $d_N\circ d_N$ is closed and exact.
\end{lemma}
\begin{proof}
Let $F^{\bullet}$ be a free resolution of $M$.  We know that $\frac{1}{N+1}(p_{N+1}(\xx)-p_{N+1}(\xx'))$ acts by zero on $M$, hence is null-homotopic when regarded as an endomorphism of $F^{\bullet}$.  This shows that $d_N$ exists.  Now, let $F'^{\bullet}$ be another free resolution, with homotopy $d_N'\in \End(F'^{\bullet})$.  Since $[[d_F,d_N],\phi]=0$, it follows that $[d_F,[d_N,\phi]]=0$.  In other words, $[d_N,\phi]$ a closed element of $\Hom(F^{\bullet},F'^{\bullet})$ of weight $q^Na\inv$.  The homology of $\Hom(F^{\bullet},F'^{\bullet})$ is isomorphic to the ext algebra $\Ext(M,M)$, which has no nonzero elements of negative $A$-degree.  This shows that $[d_N,\phi]$ is exact, as claimed.  A similar argument shows that $d_N^2$ is exact.
\end{proof}

Given a free resolution $F^{\bullet}$ and endomorphism $d_N$ as above, $d_N$ becomes closed when acting on $F^{\bullet}\otimes_{\k[\xx,\xx']} R$.  Recall that by Proposition \ref{prop: Tor invariance} the homology of this complex $F^{\bullet}\otimes_{\k[\xx,\xx']} R$ is isomorphic to $a^{-n}\HH(M)$. Since $d_N^2$ is exact on $F^{\bullet}$, we get $d_N^2=0$ on $\HH(M)$ and we can regard $d_N$ as a differential on $\HH(M)$. By Lemma \ref{lemma:res and h} it does not depend on a choice of resolution $F^{\bullet}$.

\begin{definition}\label{def:dN general}
Given $M\in \SBim_n$, let $d_N$ denote the endomorphism of $\HH(M)$ induced from the $d_N$ from Lemma \ref{lemma:res and h} on a free resolution $F$ of $M$.  Let $\HH_{\gll_N}(M):=\tw_{d_N}(\HH(M))$ denote the complex $\HH(M)$ with differential given by $d_N$.  We regard $\tw_{d_N}(\HH(M))$ as graded by the group
\[
\Z_Q\times \Z_T \ \cong \ \Z_Q\times \Z_T\times \Z_A/(2+2N,-1,-1)
\]
\end{definition}
When ``turning on'' the differential $d_N$, we must collapse the trigrading to a bigrading as in the above, since the weight of all differentials ought to be $T$, but the weight of $d_N$ is $Q^{2+2N}A\inv$.

It will be useful to have a concrete model for $\tw_{d_N}(\HH(M))$.  To this end, recall the definition of $\Delta_n(\xx,\xx')$ from \S \ref{ss:HH}. By Lemma \ref{lem: examples resolutions} $M\star \Delta_n(\xx,\xx')$ is a free resolution of $M$ over $\k[\xx,\xx']$, so we can use it as a model for $F^{\bullet}$.

Define an endomorphism $d_N$ of $\Delta_n(\xx,\xx')$ by
\begin{equation}
\label{eq: dN Koszul}
    d_N(\theta_i) = \frac{1}{N+1}h_{N}(x_i,x_i')
\end{equation}
and extend to the rest of $\Delta_n(\xx,\xx')$ via the graded Leibniz rule. Note that $\wt(d_N)=q^N a\inv$. This will not be a chain map, but $d_N^2 = 0$ and $d_{\HH}d_N + d_Nd_{\HH} =\frac{1}{N+1}( p_{N+1}(\xx) - p_{N+1}(\xx'))$ (so $d_{\HH} + d_N$ turns $\Delta_n$ into a matrix factorization; this is the $\gll_N$-matrix factorization associated to $n$ vertical arcs in \cite{RasDiff}). However, on any Soergel bimodule $M$, the difference $p_{N+1}(x_i) - p_{N+1}(x_i')$ vanishes, so $d_N$ yields a well defined chain map on $M\otimes_{\k[\xx,\xx']} \Delta_n(\xx,\xx')$.  Taking homology yields the action of $d_N$ on $\HH(M)=H\left(M\otimes_{\k[\xx,\xx']}\Delta_n(\xx,\xx')\right)$. 

\begin{example}
\label{ex: dN unknot}
For $M=R$ we get $\HH(M)=\k[\xx,\ttheta]$. The differential $d_N$ is given by
$$
d_N(\theta_i)=x_i^N
$$
which is a specialization of \eqref{eq: dN Koszul} at $\xx=\xx'$.
\end{example}

The following is clear from definitions:
\begin{proposition}\label{prop:HH with dN is functor}
For bimodules $M,N\in \SBim_n$, any bimodule map $f\colon M\to N$ defines a map $\HH(f):\HH(M)\to \HH(N)$ which commutes with $d_N$.  In other words, $\HH_{\gll_N}$ is a functor from $\SBim_n$ to the category of $\Z_Q\times \Z_A$-graded modules over the exterior algebra $\k[d_N]$ (where $d_N$ has weight $\wt(d_N)=Q^{2+2N}A\inv=q^Na\inv$).
\end{proposition}

Furthermore, $d_N$ interacts well with the multiplicative structure.

\begin{lemma}
\label{lem: bimodule dN multiplicativity}
    For Soergel bimodules $M,M'$, the multiplication map
    \begin{equation*}
        \HH(M) \otimes \HH(M') \to \HH(M \star M')
    \end{equation*}
    commutes with $d_N$, where $d_N$ acts on the left via the graded Leibniz rule.  In other words, the functor $\HH_{\gll_N}$ is lax monoidal with respect to the usual tensor product on modules over the Hopf algebra $\k[d_N]$.
\end{lemma}

\begin{proof}
We we need to prove that the diagram:
$$
\begin{tikzcd}
\Delta_n(\xx,\xx')\otimes_{\k[\xx']} \Delta_n(\xx',\xx'') \arrow{r}{m} \arrow{d}{d_N} \arrow[dotted]{dr}{\nu}& \Delta_n(\xx,\xx'') \arrow{d}{d_N}\\
\Delta_n(\xx,\xx')\otimes_{\k[\xx']} \Delta_n(\xx',\xx'') \arrow{r}{m}& \Delta_n(\xx,\xx'')
\end{tikzcd}
$$
with horizontal arrows given by multiplication and vertical by $d_N$, commutes up to homotopy. Recall that on $\Delta_n(\xx,\xx')$ we have $[d_{\HH},d_N]=p_{N+1}(\xx)-p_{N+1}(\xx')$. Therefore on $\Delta_n(\xx,\xx')\otimes_{\k[\xx']} \Delta_n(\xx',\xx'')$ we have
$$
[d_{\HH},d_N]=\frac{1}{N+1}(p_{N+1}(\xx)-p_{N+1}(\xx'))+\frac{1}{N+1}(p_{N+1}(\xx')-p_{N+1}(\xx''))=
$$
$$
\frac{1}{N+1}(p_{N+1}(\xx)-p_{N+1}(\xx'')).
$$
In particular, $[d_{\HH},d_N]$ commutes with $m$ and 
$$
[d_{\HH},d_N\circ m-m\circ d_N]=0.
$$
Now $d_N\circ m-m\circ d_N$ is a chain map of $A$-degree $-1$, but the corresponding $\Ext^{-1}$ group vanishes (there are no negative $\Ext$ between modules over polynomial rings), so $d_N\circ m-m\circ d_N$ is nullhomotopic.
\end{proof}

See also \cite{Becker} for a closely related lax monoidal functor on Soergel bimodules.

\begin{remark}
Alternatively, one can construct a nullhomotopy $\nu$ for $[d_N,m]$ explicitly by writing
$$
\nu(\theta_i\otimes \theta_j)=\delta_{i,j}h_{N-1}(x_i,x'_i,x''_i).
$$
We have  
$$
d_N(m(\theta_i\otimes 1))-m(d_N(\theta_i\otimes 1))=h_{N}(x_i,x''_i)-h_{N}(x_i,x'_i)=(x''_i-x'_i)h_{N-1}(x_i,x'_i,x''_i).
$$
On the other hand, 
$$
\nu(d(\theta_i\otimes 1))=\nu(\sum_{j\neq i}(x_j-x'_j)\theta_i\theta_j\otimes 1)-\nu(\sum_{j}(x'_j-x''_j)\theta_i\otimes \theta_j)=0-(x'_i-x''_i)h_{N-1}(x_i,x'_i,x''_i).
$$
Similarly,
$$
d_N(m(1\otimes \theta_i))-m(d_N(1\otimes \theta_i))=h_{N}(x_i,x''_i)-h_{N}(x_i',x''_i)=
$$
$$
(x_i-x'_i)h_{N-1}(x_i,x'_i,x''_i)=\nu(d(1\otimes \theta_i)).
$$
One can now check that $\nu$ extends to the map between exterior algebras such that 
$$
d_N\circ m-m\circ d_N=-[d,\nu]
$$
and the result follows. 
\end{remark}

Now we can define our version of the Rasmussen spectral sequence \cite{RasDiff}. 
\begin{definition}\label{def:slN homology}
For a complex $X\in \Ch(\SBim_n)$, we let $\CHH_{\gll_N}(X)$ be the double complex $\CHH(X)$ with horizontal differential $\HH(d_X)$ and vertical differential given by $d_N$. Let $C_{\gll_N}(X)=\mathrm{Tot}(\CHH_{\gll_N}(X))$.  Denote the homology of this complex by $H_{\gll_N}(X)$. 
\end{definition}
Note that $C_{\gll_N}(X)$ is graded by $\Z_Q\times\Z_T\times\Z_A/(2+2N,-1,-1)$.

\begin{definition}\label{def:Rasmussen SS}
 Given a complex $M\in \Ch(\SBim_n)$, the Rasmussen spectral sequence is that of a double complex, with terms given by $\CHH^{\bullet,k,i}(M)$ 
 horizontal differential $\HH(d_M)$ and vertical differential $d_N$, converging to $H_{\gll_N}(M)$.  
\end{definition}

\begin{remark}
\label{rmk: Rasmussen}
    In \cite{RasDiff}, Rasmussen originally defined these $d_N$ differently, by working with particular free resolutions $F_i^\bullet$ of each $B_i$, and building the homotopy $h$ as in Lemma~\ref{lemma:res and h} on $F_{i_1}^\bullet \star \cdots \star F_{i_n}^\bullet$ via a concrete choice of $h$ for each $F_i^\bullet$. Lemma~\ref{lem: examples resolutions} and Proposition \ref{prop: Tor invariance} now say that the complex coming from closing up $F_{i_1}^\bullet \star \cdots \star F_{i_n}^\bullet$, which in \cite{RasDiff} is called the positive homology of a MOY graph, agrees with $\CHH(M)$. Similarly, $\HH(d_M)$ agrees with Rasmussen's $d_v$, and, using Lemma~\ref{lemma:res and h}, $d_N$ agrees with Rasmussen's $d_-$. (Note that our conventions for which of the two differentials in this double complex should be drawn vertically are reversed). All of this is more-or-less the statement (see Proposition \ref{prop: Tor invariance}) that $\HH(M)$ may be computed either by resolving $M$ or by resolving the diagonal; we work with a resolution of the diagonal so that all of the constructions manifestly commute with any map of Soergel bimodules.

    In particular, while what we call $H_{\gll(N)}$ above does not obviously agree with the original Khovanov--Rozansky definition of $\gll(N)$ link homology from \cite{KR2}, it does agree with the group to which the spectral sequence from \cite{RasDiff} converges. Rasmussen's detailed arguments from \cite{RasDiff} about the collapse of this spectral sequence on a fixed Soergel bimodule (as opposed to a complex) are then needed to show that this $H_{\gll(N)}$ agrees with the version from \cite{KR2}.
\end{remark}

\begin{proposition}
\label{prop: braid dN multiplicativity}
    For braids $\beta$, $\beta'$, the multiplication map
    \begin{equation*}
        \HHH(\beta) \otimes \HHH(\beta') \to \HHH(\beta \beta')
    \end{equation*}
    commutes with all pages of the Rasmussen $\gll_N$ spectral sequence. 
\end{proposition}

\begin{proof}
The multiplication map $\HHH(\beta) \otimes \HHH(\beta') \to \HHH(\beta \beta')$ is just column-by-column the multiplication map on Hochschild homology, so commutes with both differentials of the double complex by Lemma~\ref{lem: bimodule dN multiplicativity}. Therefore, it commutes with the differential on each page of the spectral sequence.
\end{proof}

\subsection{$y$-ified homology and spectral sequences}

Next we recall the deformed (or $y$-ified) triply graded link homology from \cite{GH}.

\begin{definition}\label{def:HY}
Given $\tw_\alpha(X)\in \mathcal{Y}_{n,w}$, we let
\[
\CY(\tw_\alpha(X)):= \HH(X)\cotimes\k[\yy]/ (y_i = y_{w(i)})_{1\leq i\leq n}
\]
with differential given by applying the functor $\HH$ componentwise to the total differental $d_X+\alpha$. Then $\CY(\tw_\alpha(X))$ is a $\Z_Q\times \Z_T\times \Z_A$-graded complex; we denote the homology of this complex by $\HY(\tw_\alpha(X))$.
 If $X=F^y(\b)$ then we will abbreviate by writing $\CY(\b):=\CY(F^y(\b))$, and similarly for $\HY(\b)$.
\end{definition}
\begin{definition}
Similarly we let
\[
\CY_{\gll_N}(\tw_\alpha(X)):= C_{\gll_N}(X)\cotimes\k[\yy] / (y_i = y_{w(i)})_{1\leq i\leq n}
\]
with differential given by applying the functor $\HH_{\gll_N}$ componentwise to the total differential $d_X+\alpha$.  This is a graded by $\Z_Q\times \Z_T\times \Z_A/(2+2N,-1,-1)$.

 If $X=F^y(\b)$ then we will abbreviate by writing $\CY_{\gll_N}(\b):=\CY_{\gll_N}(F^y(\b))$, and similarly for $\HY_{\gll_N}(\b)$. 
\end{definition}

\begin{remark}\label{rmk:poly 2}
As mentioned already in Remark \ref{rem: poly 1}, if $X$ is such that $\HH^{i+2L,j-2L,k}(X)=0$ for $L\gg 0$, then every element of $\CY$ is polynomial, i.e.~ $\CY(\tw_\alpha(X)) = \HH(X)\otimes \k[y_1,\ldots,y_n]$.  This is the case whenever $X$ is a bounded complex.
\end{remark}

\begin{theorem}\label{thm:HY}[\cite{CK,GH}]
If $\b$ and $\b'$ represent the same link, then $\CY(\b)\simeq\CY(\b')$ and $\CY_{\gll_N}(\b)\simeq\CY_{\gll_N}(\b')$  up to an overall shift in tridegree.
\end{theorem}

\begin{remark}
    In fact, once it is known that $\CY(\beta)$ and $\HH_{\gll_N}(\beta)$ are link invariants, there is nothing further to check to see that $\CY_{\gll_N}$ is: since we know from Proposition~\ref{prop:HH with dN is functor} that the action of $d_N$ on complexes of Soergel bimodules is functorial, the isomorphisms $\CY(\beta) \cong \CY(\beta')$ when $\beta$ and $\beta'$ differ by a second or third Reidemeister move automatically commute with $d_N$ and so descend to $\CY_{\gll_N}$. Therefore, the only possible issue is with the first Reidemeister move, but here the relation $y_i = y_{w(i)}$ makes all terms with $y$'s drop out of the relevant part of the differential, and now the proof is identical to the one for $\HH_{\gll_N}$.
\end{remark}

\begin{proposition}\label{prop:CY is multipicative}
The constructions $X\mapsto \CY(X)$ and $X\mapsto \CY_{\gll_N}(X)$ are multiplicative  (lax monoidal) in the following sense.  Given $\tw_{\alpha}(X),\tw_{\alpha'}(X')\in \mathcal{Y}_{n,1}$ we have chain maps
\[
\CY(\tw_{\alpha}(X))\otimes \CY(\tw_{\alpha'}(X')) \to \CY(\tw_{\alpha}(X)\star \tw_{\alpha'}(X'))
\]
(and similarly for $\CY_{\gll_N}$) satisfying appopriate associativity and unit constraints.
\end{proposition}
\begin{proof}
This follows from the fact that the multiplication $\HH(M)\otimes \HH(M')\to \HH(M\star M')$ is natural in $M$ and $M'$ (Proposition \ref{prop:HH monoidal}), and commutes with with $d_N$ (Lemma \ref{lem: bimodule dN multiplicativity}).
\end{proof}

\begin{lemma}(\cite[Theorem 4.20]{GH})
\label{lem: parity}
Assume that $\k$ is a field, $\tw_{\alpha}(X)\in \mathcal{Y}_{n,w}$ and $\HHH(X)$ is supported in {\bf even} homological degrees. Then we have an isomorphism of $\k[\yy]$-modules
$$
\HY(\tw_{\alpha}(X))=\HHH(X)\cotimes\k[\yy] / (y_i = y_{w(i)})_{1\leq i\leq n} 
$$
and an isomorphism of $\k$-vector spaces
$$
\HHH(X)=\HY(X)/(y_1,\ldots,y_n).
$$
\end{lemma}

Note that \cite{GH} did not use completed tensor product since it focused on bounded complexes. Nevertheless, the same arguments apply verbatim if we use bounded above complexes and completed tensor product. 

\section{Projectors and symmetric colored homology}

\subsection{Categorified symmetrizers}
\label{ss:Pn}

We recall some constructions from \cite{Hog,idempotents}.  

\begin{definition}\label{def:Pn}
Define $\PP_n\in \Ch^-(\SBim_n)$ by the formula
\[
\PP_n \ := \ \left[\cdots \to T^{-3}\bigoplus_{i_1,i_2,i_3} B_{i_1}\star B_{i_2}\star B_{i_3}\to T^{-2}\bigoplus_{i_1,i_2} B_{i_1}\star B_{i_2} \to T\inv \bigoplus_{i_1} B_{i_1}\to \underline{R}\right]
\]
in which the indices $i_j$ range from $1$ to $n-1$, 
and the differential is an alternating sum $\sum_{j=1}^r(-1)^{j-1}\id^{\star j-1}\star b_{i_j}\star \id^{\star r-j}$. We will be often omitting the tensor product symbol $\star$ for brevity.

Let $\upsilon\colon R\to \PP_n$ denote the inclusion of the degree zero bimodule.
\end{definition}

\begin{theorem}\label{thm:P unique}(\cite{Hog,idempotents})
The complex $\PP_n$ is uniquely characterized up to homotopy equivalence in $\Ch^-(\SBim_n)$ by (1) $\PP_n\star B_i\simeq 0 \simeq B_i\star \PP_n$ for all $1\leq i\leq n-1$ and (2) there exists a chain map $\upsilon\colon R\to \PP_n$ such that $\Cone(\upsilon)$ is homotopy equivalent to a complex whose chain groups are (direct sums of shifts of) nontrivial tensor products of the $B_i$.\qedhere
\end{theorem}

The complex $\PP_n$ admits the structure of a dg algebra object in $\Ch^-(\SBim_n)$, with multiplication given componentwise by concatenation
\begin{equation}\label{eq:concatenate Bs}
(B_{i_1}\cdots  B_{i_r})\star (B_{j_1} \cdots  B_{j_s})\buildrel\cong\over\to B_{i_1} \cdots B_{i_r} B_{j_1}\cdots  B_{j_s},
\end{equation}
and unit map given by $\upsilon\colon R\to \PP_n$.

One of the main results in \cite{Hog} is the computation of $\End(\PP_n)$ and $\HHH(\PP_n)$.

\begin{theorem}\label{thm:end Pn}(\cite{Hog})
We have quasi-isomorphisms of dg algebras 
$$
\End(\PP_n;\Z)\simeq \Hom(R,\PP_n)\simeq \Z\left[\widetilde{u}_0,\ldots,\widetilde{u}_{n-1}\right]
$$ 
where the $\widetilde{u}_i$ are even variables of weight $qt^{1-i}$. Moreover, 
$$
\HHH(\PP_n;\Z)\simeq \Z\left[\widetilde{u}_0,\ldots,\widetilde{u}_{n-1},\widetilde{\xi}_0,\dots,\widetilde{\xi}_{n-1}\right]
$$ 
where the $\widetilde{\xi}_i$ are odd variables of weight $at^{1-i}$.
\end{theorem}

Next, we define a lift of $\PP_n$ to the category $\mathcal{Y}_{n,1}$.  For each linear polynomial $f\in R$, let $\eta_f\in \End(\PP_n)$ be the weight $T\inv$ endomorphism defined as follows.  Given a sequence $i_1,\ldots,i_r$ and an index $1\leq j\leq r$, let $(\eta_f)_{i_1,\ldots,i_j,\ldots,i_r}^{i_1,\ldots,i_j,i_j,\ldots,i_r}$ be the composition of maps
\[
\begin{tikzpicture}
\node (a) at (0,0) {$B_{i_1}\cdots B_{i_j}\cdots B_{i_r}$};
\node (b) at (6,0) {$B_{i_1}\cdots B_{i_j}B_{i_j}\cdots B_{i_r}$};
\node (c) at (12,0) {$B_{i_1}\cdots B_{i_j}B_{i_j}\cdots B_{i_r}$};
\path[-stealth]
(a) edge node[above] {\tiny$\id^{\star j-1}\star \Delta \star \id^{\star r-j}$} (b)
(b) edge node[above] {\tiny$\id^{\star j}\star f\star \id^{\star r-j}$} (c);
\end{tikzpicture}
\]
where the $\Delta$ here is the comultiplication $B_{i_j}\to B_{i_j}\star B_{i_j}$ and in the second map, $f$ acts by ``middle multiplication'' on $B_{i_j}\star B_{i_j}$.   Then we let $\eta_f$ be the sum of all such maps, over all sequences $i_1,\ldots,i_r\in \{1,\ldots,n-1\}$ and all indices $1\leq j\leq r$, times the sign $(-1)^{j-1}$.

 It is a straightforward exercise to verify that $[d_{\PP},\eta_f]=f(\xx)-f(\xx')$, and that the $\eta_f$ anti-commute and square to zero.  Let us abbreviate by writing $\eta_i:=\eta_{x_i}$.   Thus, we have a well defined lift
\[
\PP_n^y := \tw_{\sum_i \eta_iy_i}(\PP_n) \in \mathcal{Y}_{n,1}
\]

The deformed projector $\PP_n^y$ has the structure of a dg algebra object in $\mathcal{Y}_{n,1}$, with multiplication and units defined by the exact same formulas for $\PP_n$.

The following says that lifts of $\PP_n$ to $\mathcal{Y}_{n,1}$ are unique up to equivalence.
\begin{theorem}\label{thm:Py unique}
Suppose $\PP_n', \PP_n$ are two different models for the symmetrizer, and suppose $\tw_{\alpha}(\PP_n)$, $\tw_{\alpha'}(\PP_n')$ are lifts to objects of $\mathcal{Y}_{n,1}$. 
 Then $\tw_{\alpha}(\PP_n)\simeq \tw_{\alpha'}(\PP_n')$.
\end{theorem}
\begin{proof}
The complexes $\PP_n'$ and $\PP_n$ are homotopy equivalent by the unique characterization from Theorem \ref{thm:P unique}.  By Theorem \ref{thm:end Pn} all elements of $\Hom(\PP_n,\PP_n')$ of weight $T^{1-2r}Q^{2r}$ vanish, so by Lemma \ref{lemma:uniqueness of lifts} we get  $\tw_{\alpha}(\PP_n)\simeq \tw_{\alpha'}(\PP_n')$.
\end{proof}

The following are some useful properties of $\PP_n^y$.
\begin{proposition}\label{prop:Py is idempt}
The $y$-ified symmetrizer $\PP_n^y$ is a unital idempotent in $\mathcal{Y}_{n,1}$.
\end{proposition}
\begin{proof}
We must show that $\PP_n^y\star \Cone(\one_n\to \PP_n^y)\simeq 0 \simeq \Cone(\one_n\to \PP_n^y)\star \PP_n^y$.  Both of these homotopy equivalences follow from their undeformed versions \cite{Hog}, by Lemma \ref{lemma:lift of equiv}.
\end{proof}

\begin{proposition}\label{prop:Py is central}
$\PP_n^y$ has the structure of an object in the Drinfeld center of the triangulated monoidal category $\bigoplus_w H^0(\mathcal{Y}_{n,w})$.
\end{proposition}
\begin{proof}
For any $X^y\in \mathcal{Y}_{n,w}$ we have homotopy equivalences
\[
X^y\star \PP_n^y\rightarrow \PP_n^y\star X^y\star \PP_n^y\leftarrow \PP_n^y\star X^y
\]
given by applying the unit map of $\PP_n^y$ on the left (respectively right) of $X^y$ as in \cite[Theorem 4.23]{idemptriang}. These homotopy equivalences are obviously natural in $X^y$.
\end{proof}

\begin{corollary}
Let $X_1^y,X_2^y\in \mathcal{Y}_{n,w}$ be such that $X_2^y\star \PP^y_n\simeq X_2^y$.  Then
\[
\Hom(X_1^y\star \PP^y_n , X_2^y)\simeq \Hom(X_1^y, X_2^y)
\]
\end{corollary}
\begin{proof}
See \cite[Proposition 4.16]{idemptriang}.
\end{proof}
In the special case $X_1^y=X_2^y=R$, we have the following.

\begin{lemma}\label{lem:end Pn y}
We have quasi-isomorphisms of dg algebras 
$$
\End_{\mathcal{Y}_{n,1}}(\PP_n^y;\Z)\simeq \Hom_{\mathcal{Y}_{n,1}}(R,\PP_n^y)\simeq \Z\left[\yy,\widetilde{u}_0,\ldots,\widetilde{u}_{n-1}\right]
$$ 
where the $\widetilde{u}_i$ are even variables of weight $qt^{1-i}$. Moreover, 
$$
\HY(\PP_n^y;\Z)\simeq \Z\left[\yy,\widetilde{u}_0,\ldots,\widetilde{u}_{n-1},\widetilde{\xi}_0,\dots,\widetilde{\xi}_{n-1}\right]
$$ 
where the $\widetilde{\xi}_i$ are odd variables of weight $at^{1-i}$.
\end{lemma}

\begin{proof}
We prove the last statement about $\HY$, the proofs of other statements are identical.
By Lemma \ref{lem: parity} and Theorem \ref{thm:end Pn} we have
\begin{equation}
\label{eq: HY Pn completed}
\HY(\PP_n^y;\Z)\simeq \Z\left[\widetilde{u}_0,\ldots,\widetilde{u}_{n-1},\widetilde{\xi}_0,\dots,\widetilde{\xi}_{n-1}\right]\cotimes\Z[\yy].
\end{equation}
It remains to check that the completed tensor product with $\Z[\yy]$ coincides with the usual one.
Let us describe the subspace of homogeneous elements with weight $a^{k}q^{m}t^{s}$ in \eqref{eq: HY Pn completed}. Recall that $\widetilde{u}_i$ have weight $qt^{-i}$ and $\widetilde{\xi}_i$ have weight $at^{-i}$. The $(a,q)$-weights determine the monomials in $\widetilde{\uu}$ and $\widetilde{\xx_i}$ which might appear, and the degree in $\yy$ is determined from the $t$-weight. 

Therefore the homogeneous subspace
is spanned by the expressions
\begin{equation}
\label{eq: basis HY}
\widetilde{u}_{i_1}\cdots \widetilde{u}_{i_m}\widetilde{\xi}_{j_1}\cdots \widetilde{\xi}_{j_k}g(\yy)
\end{equation}
where $g(\yy)$ is a homogeneous polynomial of degree $i_1+\ldots+i_m+j_1+\ldots+j_k+s$. Here $k,m\ge 0$ and $s\in \Z$, but we sum over terms with
$i_1+\ldots+i_m+j_1+\ldots+j_k\ge -s$.
In particular, this subspace  is finite-dimensional for fixed $k,m,s$ and 
$$
\Z[\widetilde{\uu},\widetilde{\xxi}]\cotimes\Z[\yy]=\Z[\yy,\widetilde{\uu},\widetilde{\xxi}].
$$
\end{proof}

\subsection{Symmetrizers and Rouquier complexes}
\label{ss:symmetrizer and Rouquier}

Let $e\colon \Br_n\to \Z$ be the group homomorphism sending $\sigma_i^{\pm}\mapsto \pm 1$ for all $i$.

\begin{definition}
Let $\sT(\b):=(TQ\inv)^{-e(\b)} F(\b)$ and $\sTy(\b):=(TQ\inv)^{-e(\b)} F^y(\b)$.
\end{definition}

Note that the unique copy of $R$ sits in degree zero in $\sT(\b)$ and $\sTy(\b)$. Also, it is clear that
$$
\sTy(\beta)\star \sTy(\gamma)=\sTy(\beta\gamma).
$$

\begin{proposition}\label{prop:twisted projectors}
There exists an object $\PP_{n,w}^y\in \mathcal{Y}_{n,w}$, depending on a choice of permutation $w\in S_n$, with the property that $\sT^y(\b)\star \PP_n^y\simeq  \PP_{n,w}^y$ for any braid $\b\in \Br_n$ with underlying permutation $w$.
\end{proposition}

\begin{proof}
Suppose that $\beta=\beta_1\sigma_i\beta_2$ and $\beta'=\beta_1\sigma_i^{-1}\beta_2$, it is sufficient to show that $\sTy(\beta)$ and $\sTy(\beta')$ become homotopy equivalent after tensoring with $\PP_n^y$.
For this, recall the map $\rho_i\colon (TQ\inv) \yTii\to(TQ\inv)\inv \yTi$ from Lemma \ref{lem: skein}.  The cone of this map is homotopy equivalent to a deformation of a complex constructed from two copies of $B_i$. We get a map
$$
\Id_{\beta_1}\star \rho_i\star \Id_{\beta_2}:\sTy(\beta')\to \sTy(\beta)
$$
with the cone filtered by $F^y(\beta_1)\star B_i\star F^y(\beta_2)$. Therefore $\Cone(\Id_{\beta_1}\star \rho_i\star \Id_{\beta_2})$ belongs to the tensor ideal generated by $B_i$ and becomes contractible after tensoring with $\PP_n^y$: this follows from Lemma \ref{lemma:lift of equiv} together with the annihilating properties of $\PP_n$ from Theorem \ref{thm:P unique}.  This shows that $\Id_{\beta_1}\star \rho_i\star \Id_{\beta_2}$ becomes a homotopy equivalence after tensoring with $\PP_n^y$, and the proposition follows.
\end{proof}

Proposition \ref{prop:twisted projectors} says that $\sTy(\b)\star \PP_n^y\simeq \PP_{n,w}^y$ depends only on the permutation $w$ represented by $\b$, up to homotopy equivalence.  Precomposing with the unit map of $\PP_n^y$ gives us a map $\upsilon_{\b}:\sTy(\b)\to \PP_{n,w}^y$. More precisely, we get the following result.

\begin{theorem}\label{thm:map of diagrams}
There is a family of maps $\mu_{w,v}\colon \PP^y_{n,w}\star \PP^y_{n,v}\to \PP^y_{n,wv}$, $\upsilon_\b\colon \sTy(\b)\to \PP_{n,w}^y$ such that 
\begin{enumerate}
\item $\upsilon_{\one_n}=\upsilon_n$. 
\item the $\mu_{w,v}$ are homotopy equivalences.
\item the $\mu_{w,v}$ satisfy the associative axiom up to homotopy.
\item the $\mu_{w,v}$ satisfy the unit axiom up to homotopy, with unit given by $\upsilon_n$.
\item the following diagram commutes up to homotopy for any braids $\b,\gamma\in \Br_n$ with underlying permutations $w,v\in S_n$:
\[
\begin{tikzpicture}
\node (a) at (0,0) {$\sTy(\b)\star \sTy(\gamma)$};
\node (b) at (4,0) {$\sTy(\b\gamma)$};
\node (c) at (0,-2) {$\PP^y_{n,w}\star \PP^y_{n,v}$};
\node (d) at (4,-2) {$\PP^y_{n,wv}$};
\path[-stealth]
(a) edge node[above] {} (b)
(c) edge node[above] {$\mu_{w,v}$} (d)
(a) edge node[left] {$\upsilon_\b\star \upsilon_\gamma$} (c)
(b) edge node[right] {$\upsilon_{\b\gamma}$} (d);
\end{tikzpicture}
\]
\end{enumerate}
\end{theorem}
\begin{proof}
Fix a permutation $w\in S_n$. Let $\pi\colon \Br_n\to S_n$ denote the canonical projection. The complexes $\sTy(\b)\star \PP^y_n$ with $\b\in \pi\inv(w)$  are all homotopy equivalent by Proposition \ref{prop:twisted projectors}.  If $\b,\b'\in\pi\inv(w)$ then the hom complex from $\sTy(\b)\star \PP^y_n$ to  $\sTy(\b')\star \PP^y_n$ is homotopy equivalent to
\[
\Hom\left(\sTy(\b)\star \PP^y_n, \sTy(\b')\star \PP^y_n\right) \simeq \Hom\left(\PP^y_n, \sTy(\b\inv \b')\star \PP^y_n\right)\simeq 
\]
\[
\End(\PP^y_n)\simeq \k[\yy,\widetilde{u}_0,\cdots,\widetilde{u}_{n-1}]
\]
where we have used the fact that $\b\inv\b'$ is a pure braid, as well as Lemma \ref{lem:end Pn y}.  In particular, in bidegree $(0,0)$ this hom space is $\k$.  Thus, any two homotopy equivalences $\phi_1,\phi_2\colon\sTy(\b)\star \PP^y_n\simeq \sTy(\b')\star \PP^y_n$ must be homotopic up to an invertible scalar.  We can fix the scalar by requiring that the unique copies of $R$ be mapped to each other by $\id$.

Now, we compute if $\b\in \pi\inv(w)$ and $\gamma\in \pi\inv(v)$ then
\begin{align*}
\PP^y_{n,w}\star \PP^y_{n,v}&=\sTy(\b_w)\star \PP^y_n \star \sTy(\b_v)\star \PP^y_n\\
&\simeq \sTy(\b)\star \sTy(\gamma)\star \PP^y_n\star \PP^y_n \simeq  \sTy(\b\gamma)\star \PP^y_n \simeq \PP^y_{n,wv}
\end{align*}
When in the second line we used the centrality and idempotence of $\PP^y_n$.  All of the above homotopy equivalences are unique in their bidegree, up to homotopy and invertible scalar.  The invertible scalar can be fixed by requiring, as always, that the unique copies of $R$ be mapped to each other via $\id$.  This defines $\mu_{w,v}$.

The asserted properties are all straightforward to check.
\end{proof}

\begin{remark}
Un-$y$-ified, the maps $\mu_{w,v}$ are all the standard multiplication on $\PP_n$, and the maps $\upsilon_\b$ are easily defined by explicit formulas.
\end{remark}

\subsection{Colored homology: definition and invariance}
\label{sec: colored homology}
We can use the projectors $\PP_n^y$ to defined $y$-ified symmetric-colored link homology following \cite{CooperKrushkal} (see also \cite{Cautis,Hog,HM,Rozansky}).

Fix a braid $\b\in \Br_m$, and let $w\in S_m$ be the permutation represented by $\b$.  Let $L=\hat{\b}$ be the braid closure. The set of components of $L$ can be identified with $\{1,\ldots,m\}/\sim$, where $\sim$ is the equivalence relation generated by $i\sim w(i)$ for all $i$. For each $1\leq i\leq m$ let $[i]$ denote the equivalence class containing $i$ (i.e.~the component of $\hat{\b}$ containing the $i$-th strand of $\b$).

Fix a sequence $\vec{n}=(n_1,\ldots, n_m)\in \Z_{\geq0}^m$.  We say that $\vec{n}$ is \emph{closeable} if $n_i=n_{w(i)}$ for all $i$; in this case we write $\mathbf{n}\colon \pi_0(L)\to\Z_{\geq0}$ for the induced map.  

Let $\beta^{(n_1,\ldots,n_m)}$ be the braid on $\sum_{i=1}^{m}n_i$ strands obtained by cabling $\beta$: the strand indexed by $i$ is replaced by $n_i$ parallel copies of the same strand (with respect to the blackboard framing).   For example, for the elementary Artin braid generator $\sigma_1\in \Br_2$ we have
\[
\sigma_1^{n_1,n_2} := (\sigma_{n_1}\cdots \sigma_2\sigma_1) \cdots (\sigma_{n_1+n_2}\cdots \sigma_{2+n_2}\sigma_{1+n_2})  = 
\begin{tikzpicture}[scale=1.5,baseline=.5cm]
\draw[very thick]
(1,0) -- (0,1)
(1.3,0) -- (.3,1)
(1.6,0) -- (.6,1);
\draw[line width=5pt,white]
(0,0) -- (1,1)
(.2,0) -- (1.2,1)
(.4,0) -- (1.4,1)
(.6,0) -- (1.6,1);
\draw[very thick]
(0,0) -- (1,1)
(.2,0) -- (1.2,1)
(.4,0) -- (1.4,1)
(.6,0) -- (1.6,1);
\node at (0,1.1) {\scriptsize$1$};
\node at (.3,1.1) {\scriptsize$\cdots$};
\node at (.6,1.1) {\scriptsize$n_1$};
\node at (1,1.1) {\scriptsize$1$};
\node at (1.3,1.1) {\scriptsize$\cdots$};
\node at (1.6,1.1) {\scriptsize$n_2$};
\end{tikzpicture}
\in \Br_{n_1+n_2}
\]

If $\vec{n}$ is closeable, then the closure of $\beta^{\vec{n}}$ is the $\mathbf{n}$-cable of $L$.

\begin{definition}
\label{def: symmetric colored homology}
Let $\b\in \Br_m$ and let $\vec{n}=(n_1,\ldots,n_m)\in \Z_{\geq 0}^m$ be closeable. We define $S^{\mathbf{n}}$-colored $y$-ified homology of $L=\hat{\b}$ as follows:
$$
\CY_{S^{\vec{n}}}(\b) := \CY\left[\left(\PP^y_{n_1}\boxtimes \cdots \boxtimes \PP^y_{n_m}\right)\star F^y\left(\beta^{\vec{n}}\right)\right]
$$
and
$$
\CY_{\gll_N,S^{\vec{n}}}(\b) := \CY_{\gll_N}\left[\left(\PP^y_{n_1}\boxtimes \cdots \boxtimes \PP^y_{n_m}\right)\star F^y\left(\beta^{\vec{n}}\right)\right]
$$
Denote the homologies of these complexes by $\HY_{S^{\vec{n}}}(\b)$ and $\HY_{\gll_N,S^{\vec{n}}}(\b)$.
\end{definition}

\begin{theorem}
\label{thm: colored homology invariant}
The homologies $\HY_{S^{\vec{n}}}(\b)$ and $\HY_{\gll_N,S^{\vec{n}}}(\b)$ depend only on the oriented link $L=\hat{\b}$ up to isomorphism and overall shift.
\end{theorem}

\begin{proof}
We focus entirely on the triply graded colored homology, as the $\gll_N$ version is similar.  First, observe that the cable $\b^{\vec{n}}$ is invariant under the braid relations applied to $\b$.  To prove the theorem we need to show that our construction is invariant under the Markov moves.

For invariance under the Markov 1 move (i.e.~invariance under conjugation $\b\mapsto \gamma\b\gamma\inv$) we need to show that
\begin{equation}\label{eq:sliding projectors}
\left(\PP^y_{v(n_1)}\boxtimes \cdots \boxtimes \PP^y_{v(n_m)}\right)\star F^y(\gamma^{\vec{n}}) \simeq F^y(\gamma^{\vec{n}})\star \left(\PP^y_{n_1}\boxtimes \cdots \boxtimes \PP^y_{n_m}\right)
\end{equation}
whenever $\gamma$ is a braid with underlying permutation $v\in S_m$.  The un-$y$-ified version of the homotopy equivalence \eqref{eq:sliding projectors} is proven by standard techniques.  Indeed, the general result in \cite[Corollary 1.10]{EHdrinfeld}, independently proved by different methods in \cite{LMRSW,SW}, implies that the cabled Rouquier complex  $F(\gamma^{\vec{n}})$ is in the $A_\infty$ Drinfeld centralizer of $\SBim_{n_1}\otimes \cdots \otimes \SBim_{n_m}$ inside $\Ch(\SBim_{n_1+\cdots+n_m})$ (with the standard right action and left action twisted by the permutation $v$).  Thus, any complex of objects in $\SBim_{n_1}\otimes \cdots \otimes \SBim_{n_m}$ can be slid past the cabled Rouquier complex $F(\gamma^{\vec{n}})$.  

Since $F(\gamma^{\vec{n}})$ is invertible, we get
\[
\End\left(\left(\PP_{n_1}\boxtimes \cdots \boxtimes \PP_{n_m}\right)\star F(\gamma^{\vec{n}})\right) \simeq \End\left(\PP_{n_1}\boxtimes \cdots \boxtimes \PP_{n_m}\right) \simeq \bigotimes_{i=1}^m \End(\PP_{n_i})
\]
and \eqref{eq:sliding projectors} now follows from Theorem \ref{thm:end Pn} and Lemma \ref{lemma:lift of equiv}.

For invariance under the Markov 2 move, we need to show that
\begin{equation}
\CY\left((\PP^y_{n_1}\boxtimes\cdots\boxtimes \PP^y_{n_m}\boxtimes \PP^y_{n_m}) \star F^y((\sigma_m \b)^{(\vec{n},n_m)} \right) \simeq \CY\left((\PP^y_{n_1}\boxtimes\cdots\boxtimes \PP^y_{n_m}) \star F^y(\b^{\vec{n}}) \right).
\end{equation}
where $(\vec{n},n_m)=(n_1,\ldots,n_m,n_m)$.
One constructs the above homotopy equivalence as a composition of homotopy equivalences corresponding to local moves of the form
\[
\begin{tikzpicture}[scale=1.5,baseline=.5cm]
\draw[very thick]
(1,-.65)--(1,0) -- (0,1)--(0,1.65)
(1.2,-.65)--(1.2,0) -- (.2,1)--(.2,1.65)
(1.4,-.65)--(1.4,0) -- (.4,1)--(.4,1.65)
(1.6,-.65)--(1.6,0) -- (.6,1)--(.6,1.65);
\draw[line width=5pt,white]
(0,0) -- (1,1)
(.2,0) -- (1.2,1)
(.4,0) -- (1.4,1)
(.6,0) -- (1.6,1);
\draw[very thick]
(0,-.65)--(0,0) -- (1,1)--(1,1.65)
(.2,-.65)--(.2,0) -- (1.2,1)--(1.2,1.65)
(.4,-.65)--(.4,0) -- (1.4,1)--(1.4,1.65)
(.6,-.65)--(.6,0) -- (1.6,1)--(1.6,1.65);
\filldraw[fill=white] (-.1,1) rectangle (.7,1.5);
\filldraw[fill=white] (.9,1) rectangle (1.7,1.5);
\node at (0,1.8) {\scriptsize$1$};
\node at (.3,1.8) {\scriptsize$\cdots$};
\node at (.6,1.8) {\scriptsize$n_m$};
\node at (1,1.8) {\scriptsize$1$};
\node at (1.3,1.8) {\scriptsize$\cdots$};
\node at (1.6,1.8) {\scriptsize$n_m$};
\node at (.4,1.25) {\small$\PP^y_{n_m}$};
\node at (1.4,1.25) {\small$\PP^y_{n_m}$};
\end{tikzpicture}
\buildrel\simeq\over \rightarrow 
\begin{tikzpicture}[scale=1.5,baseline=.5cm]
\draw[very thick]
(1,-.65)--(1,0) -- (0,1)--(0,1.65)
(1.2,-.65)--(1.2,0) -- (.2,1)--(.2,1.65)
(1.4,-.65)--(1.4,0) -- (.4,1)--(.4,1.65)
(1.6,-.65)--(1.6,0) -- (.6,1)--(.6,1.65);
\draw[line width=5pt,white]
(0,0) -- (1,1)
(.2,0) -- (1.2,1)
(.4,0) -- (1.4,1)
(.6,0) -- (1.6,1);
\draw[very thick]
(0,-.65)--(0,0) -- (1,1)--(1,1.65)
(.2,-.65)--(.2,0) -- (1.2,1)--(1.2,1.65)
(.4,-.65)--(.4,0) -- (1.4,1)--(1.4,1.65)
(.6,-.65)--(.6,0) -- (1.6,1)--(1.6,1.65);
\filldraw[fill=white] (-.1,1) rectangle (.7,1.5);
\filldraw[fill=white] (-.1,-.5) rectangle (.7,0);
\node at (0,1.8) {\scriptsize$1$};
\node at (.3,1.8) {\scriptsize$\cdots$};
\node at (.6,1.8) {\scriptsize$n_m$};
\node at (1,1.8) {\scriptsize$1$};
\node at (1.3,1.8) {\scriptsize$\cdots$};
\node at (1.6,1.8) {\scriptsize$n_m$};
\node at (.4,1.25) {\small$\PP^y_{n_m}$};
\node at (.4,-.25) {\small$\PP^y_{n_m}$};
\end{tikzpicture}
\buildrel\simeq\over \rightarrow  
\begin{tikzpicture}[scale=1.5,baseline=.5cm]
\draw[very thick]
(0,-.65)--(0,1.65)
(.2,-.65)--(.2,1.65)
(.4,-.65)--(.4,1.65)
(.6,-.65)--(.6,1.65);
\filldraw[fill=white] (-.1,1) rectangle (.7,1.5);
\filldraw[fill=white] (-.1,.25) rectangle (.7,.75);
\filldraw[fill=white] (-.1,-.5) rectangle (.7,0);
\node at (0,1.8) {\scriptsize$1$};
\node at (.3,1.8) {\scriptsize$\cdots$};
\node at (.6,1.8) {\scriptsize$n_m$};
\node at (.4,1.25) {\small$\PP^y_{n_m}$};
\node at (.4,-.25) {\small$\PP^y_{n_m}$};
\node at (.32,.5) {\small$\FT_{n_m}$};
\end{tikzpicture}
\buildrel\simeq\over \rightarrow  
\begin{tikzpicture}[scale=1.5,baseline=.5cm]
\draw[very thick]
(0,-.65)--(0,1.65)
(.2,-.65)--(.2,1.65)
(.4,-.65)--(.4,1.65)
(.6,-.65)--(.6,1.65);
\filldraw[fill=white] (-.1,.25) rectangle (.7,.75);
\node at (0,1.8) {\scriptsize$1$};
\node at (.3,1.8) {\scriptsize$\cdots$};
\node at (.6,1.8) {\scriptsize$n_m$};
\node at (.4,.5) {\small$\PP^y_{n_m}$};
\end{tikzpicture}
\]
Each of these moves corresponds to a homotopy equivalence of complexes $\CY$ (up to shift): the first homotopy equivalence follows from \eqref{eq:sliding projectors}, the second is the Markov 2 invariance of $\CY$, and the last is an application of Proposition \ref{prop:twisted projectors} (note that the full twist braid is a pure braid, so  $\FT_{n_m}\star \PP_{n_m}^y$ is homotopy equivalent to $\PP_{n_m}^y$ up to a shift) followed by the homotopy equivalence $\PP_n^y\star \PP_n^y\simeq \PP_n^y$ from Lemma \ref{prop:Py is idempt}. This completes the proof.
\end{proof}

\begin{remark}
One should compare the above construction of projector-colored homology with the construction of Cooper--Krushkal \cite{CooperKrushkal}. The projectors $\PP_n^y$ satisfy the same kind of local properties satisfied by the Cooper-Krushkal projectors (idempotence, sliding over and under crossings, absorption of crossings), and so a proof strategy similar to \cite[Section 5]{CooperKrushkal} would work here as well. But our construction and proof of invariance look a bit different from \cite{CooperKrushkal} on the surface, since we start with a braid presentation and place all projectors at the top of the cabled braid.
\end{remark}

\begin{remark}
Recall that $\mathbf{n}\colon \pi_0(L)\to \Z_{\geq 0}$ denotes the function induced by $\vec{n}=(n_1,\ldots,n_m)$.  One can show that the product of symmetric groups $\prod_{c\in \pi_0(L)} S_{\mathbf{n}(c)}$ acts on the colored complex $\CY_{S^{\mathbf{n}}}(\b)$ (regarded as an object of a suitable derived category).   Roughly speaking, a permutation $x\in S_{\mathbf{n}(c)}$ acts as a composition:
\begin{enumerate}
    \item produce a canceling pair of Rouquier complexes $F^y(\gamma_x)\star F^y(\gamma_x\inv)$ on the appropriate component of $L$, where $\gamma_x$ denotes a positive braid lift of $x$.
    \item slide $F^y(\gamma_x\inv)$ around the entire link component; when sliding  $F^y(\gamma_x\inv)$ past a projector $\PP^y_n$ we use centrality of $y$-ified projectors.
    \item cancel the pair $F^y(\gamma_x\inv)\star F^y(\gamma_x)$.
\end{enumerate}
This action factors through the symmetric group (up to homotopy) by arguments similar to \cite{GLW,GHW}.  Taking the invariants with respect to this symmetric group action would yield a link invariant that might also be referred to $S^{\mathbf{n}}$-colored homology.
\end{remark}

\begin{remark}
In \cite[Theorem 5.24]{Conners} Conners described a different  model of $\PP_n^y$ (compare with Theorem \ref{thm:Pn as colimit} below) and conjectured in \cite[Conjecture 5.26]{Conners} that it satisfies all desired properties. Theorem \ref{thm: colored homology invariant} resolves this conjecture in the positive.
\end{remark}

\section{From full twists to colored homology}

\subsection{Powers of the full twist}
\label{ss:powers of ft}

\begin{definition}
Let $\htt_n= (\sigma_1)(\sigma_2\sigma_1)\cdots (\sigma_{n-1}\cdots \sigma_2\sigma_1)\in \Br_n$ denote the half-twist braid, and let $\ft_n = \htt_n^2$ denote the full twist.
\end{definition}

For example, $\htt_2=\sigma_1$ and $\ft_2=\sigma_1^2$.

\begin{definition}
We define by $\FT_n=F(\ft_n)$ the Rouquier complex for the full twist braid, and by $\FTy_n=F^y(\ft_n)$ the corresponding $y$-ified Rouquier complex.
\end{definition}

\begin{definition}
Let $\Sigma_n:=(T^2Q^{-2})^{\binom{n}{2}}=t^{\binom{n}{2}}$, and let $\a\colon \one\to \Sigma_n\inv\FTy_n$ given by the inclusion of the right-most chain bimodule, see Lemma \ref{lem: rightmost R y-ified}.

For each $k\in \Z_{\geq 0}$ let $\a_k\colon \Sigma_n^{-k}\FTy_n^k\to \Sigma_n^{-k-1}\FTy_n^{k+1}$ be the chain map defined by $\a \star \Id_{\FTy_n^k}$.
\end{definition}

\begin{theorem}\label{thm:Pn as colimit}
We have
$$
\PP_n^y\simeq \hocolim_{k\to \infty}\left(\Sigma_n^{-k}\FTy_n^k,\a_k\right),
$$
where the homotopy colimit is taken inside the dg category $\mathcal{Y}_{n,1}$. The unit map $\upsilon:R\to \PP_n^y$ is given by the inclusion of $\FT^0=R$ in the homotopy colimit.
\end{theorem}
\begin{proof}
For a concrete description of the stated homotopy colimit $\PP_n^y$ we use the mapping telescope construction.  Consider the countable direct sum $\mathbf{L}=\bigoplus_{k\geq 0} \Sigma_n^{-k} \FTy_n^k$.  Let $A\in \End(\mathbf{L})$ be the degree zero closed endomorphism whose components are the $\a_k$.  Then we set
\begin{equation}\label{eq:Py}
\hocolim_{k\to \infty}(\Sigma_n^{-k}\FTy_n^k,\a_k):= \Cone(\mathbf{L}\buildrel A-\id \over \longrightarrow\mathbf{L}) = 
\left( \begin{tikzcd}
T\inv\FTy_n^0 \arrow{r}{-\Id} \arrow{dr}{\a_0} & \FTy_n^0\\
T\inv\Sigma_n\inv\FTy_n^1 \arrow{r}{-\Id}   \arrow{dr}{\a_1}& \Sigma_n\inv\FTy_n^1\\
T\inv \Sigma_n^{-2}\FTy_n^2 \arrow{r}{-\Id}   \arrow{dr}{\a_2}& \Sigma_n^{-2}\FTy_n^2\\
\vdots & \vdots
\end{tikzcd}
\right). 
\end{equation}
It is proven in \cite{Hog} that the undeformed projector $\PP_n$ is the homotopy colimit of powers of undeformed full twists, with maps given by $\a_k$ (mod $\yy$).  Theorem \ref{thm:Py unique} then completes the proof.
\end{proof}

Combining the above description of $\PP_n^y$ with results in \cite{GH} will yield a useful presentation of $\HY(\PP_n^y)$.  We begin by recalling some results on homology of the powers of the full twist following \cite{GH}.

\begin{definition}
We denote by  $\one_n$ the $n$-strand identity braid, its closure is the $n$-component unlink $O_n$. We have
$$   
\HY(\one_n)=\k[\xx,\yy,\ttheta]
$$
with $\Z^3$-grading determined by \eqref{eq:weights of xyt}.
There is an action of $S_n$ permuting $x_i,y_i$ and $\theta_i$ simultaneously, and we define $\cJ_n$ as the ideal in $\HY(\one_n)$ generated by the antisymmetric polynomials under this action.
\end{definition}

\begin{proposition}(\cite[Proposition 6.1]{GH})
\label{prop: psi}
We have 
\begin{equation}
\label{eq: Psi}
\Hom_{\mathcal{Y}_{n,1}}(\FTy_n,\one_n)\simeq \k[\xx,\yy].
\end{equation}
The generator $\Psi$ of \eqref{eq: Psi} is given by the composition of crossing change maps $\psi_i$, as in Lemma \ref{lem: rightmost R y-ified}.
\end{proposition}

\begin{definition}\label{def:splitting map 2}
Let $\Psi\colon \HY(\FTy_n)\to \HY(\one_n)$ denote the map in $\HY$ induced by the generator of \eqref{eq: Psi}.  
By abuse of notation, we denote the tensor powers $\Psi^{\star k} \colon \HY(\FTy_n^k)\to \HY(\one_n)$ also by $\Psi$.  We refer to the maps $\Psi$ and $\Psi^{\star k}$ as to \emph{splitting maps}.
\end{definition}

\begin{theorem}\cite{GH}
\label{thm:GH ft}
Assume $\k=\C$.
The map $\Psi\colon \HY(\FTy_n^k;\C)\to \HY(\one_n;\C)$ is  injective, with image equal to the ideal  $\cJ^k_n\subset \HY(\one_n;\C)$, so that $\HY(\FTy^k_n;\C)\cong \cJ^k_n$. 
These isomorphisms are compatible with the multiplicative structures, in the sense that the diagram
\[
\begin{tikzpicture}
\node (a) at (0,0) {$\HY(\FTy_n^k;\C)\otimes \HY(\FTy_n^m;\C)$};
\node (b) at (5,0){$\HY(\FTy_n^{k+m};\C)$};
\node (c) at (0,-2) {$\cJ_n^k\otimes \cJ_n^m$};
\node (d) at (5,-2) {$\cJ_n^{k+m}$};
\path[-stealth]
(a) edge node[above] {} (b)
(c) edge node[above] {} (d)
(a) edge node[left] {} (c)
(b) edge node[right] {} (d);
\end{tikzpicture}
\]
commmutes for all $k,m\geq 0$.
\end{theorem}
 
Next, we identify the connecting maps $\a_k$ using Theorem \ref{thm:GH ft}.

\begin{lemma}
\label{lem: connecting map Delta}
The maps $\Psi$ and $\rho$ are independent of the factorization in Lemma \ref{lem: rightmost R y-ified}. Furthermore, we have the commutative diagram
\begin{equation}
\label{eq: psi and alpha}
\begin{tikzcd}
\Sigma_n^{-k-1}\HY(\FTy_n^{k+1};\C) \arrow{r}{\Psi^{\star(k+1)}}& \Sigma_n^{-k-1}\cJ_{n}^{k+1}\\
\Sigma_n^{-k}\HY(\FTy_n^k;\C) \arrow{r}{\Psi^{\star k}} \arrow{u}{\a_k}& \Sigma_n^{-k}\cJ_{n}^{k}\arrow{u}{\Delta_n}
\end{tikzcd}
\end{equation}
where  $\Delta_n=\prod_{1\le i<j\le n} (y_i-y_j)$ is the Vandermonde determinant.
\end{lemma}

\begin{proof}
The splitting map $\Psi$ is independent of factorization by Proposition \ref{prop: psi}, let us prove that $\a$ is independent as well.
We can identify $\a\in \Hom_{\mathcal{Y}_{n,1}}(\one_n,\FTy_n)$ with the corresponding class in $\HY^{a=0}(\FTy_n;\C)$ of bidegree $q^{0}t^{n(n-1)/2}$. 

By Theorem \ref{thm:GH ft} the $(q,t)$ bidegree on $\HY^{a=0}(\FTy_n;\C)$ matches the $(\xx,\yy)$ bidegree on $\cJ_n^{a=0}\subset \C[\xx,\yy]$. The ideal $\cJ_n^{a=0}$ is concentrated in nonnegative $q$-degrees and its homogeneous components of $q$-degree zero are generated over $\C[\yy]$ by antisymmetric polynomials in $\C[\yy]$, hence generated by $\Delta_n$.  Therefore  $\Delta_n$ is the unique (up to a scalar) antisymmetric polynomial of bidegree $q^{0}t^{n(n-1)/2}$, which corresponds to $\a$. 

By Lemma \ref{lem: rightmost R y-ified} we get $\Psi\circ \a=\Delta_n$.
Next, we get the identity
$$
\Psi^{\star(k+1)}\circ \a_k=\Psi^{\star(k+1)}\circ (\a \star \Id_{\FT_n^k})=(\Psi\circ\a)\star (\Psi^{\star k}\circ \Id_{\FT_n^k})=(\Delta_n \Id_{\one})\star \Psi^{\star k}=\Delta\Psi^{\star k}.
$$
which implies the following commutative diagram:
$$
\begin{tikzcd}
\Sigma_n^{-k-1}\FTy_n^{k+1} \arrow{r}{\Psi^{\star(k+1)}}& \Sigma_n^{-k-1}\one\\
\Sigma_n^{-k}\FTy_n^k \arrow{r}{\Psi^{\star k}} \arrow{u}{\a_k}& \Sigma_n^{-k}\one\arrow{u}{\Delta_n}
\end{tikzcd}
$$
By applying $\HY$ to this diagram, we get \eqref{eq: psi and alpha}.
\end{proof}

\begin{definition}
Let $\cA_n=\bigoplus_{k=0}^{\infty}s^k\cJ_n^k$, regarded as a graded subalgebra of $\C[\xx,\yy,\ttheta,s]$.  An element of $\cJ_n^k$ will be regarded as an element of $\cA_n$ of \emph{$s$-degree $k$}.
\end{definition}

We will consider a certain localization of $\cA_n$ momentarily, but since  $\cA_n$ has zero divisors, we need to be careful.

\begin{lemma}
\label{lem: not zero divisor}
The polynomial $\Delta_n\in \cJ_n$  (considered as an $s$-degree one element in $\cA_n$) is not a zero divisor in $\cA_n$.
\end{lemma}
\begin{proof}
 Indeed, $\HY(\one_n;\C)=\C[\xx,\yy,\ttheta]$ is a free module of rank $2^n$ over 
$\C[\xx,\yy]$, so multiplication by $\Delta_n$ is injective on $\HY(\one_n;\C)$. Since $\cJ_n^k$ is a submodule of $\HY(\one_n;\C)$, the multiplication by $\Delta_n$ is injective on $\cJ_n^k$ as well.
\end{proof}

\begin{corollary}
The map $\a_k: \Sigma_n^{-k}\HY(\FTy_n^k;\C)\to \Sigma_n^{-k-1}\HY(\FTy_n^{k+1};\C)$ is injective.\qed
\end{corollary}

Next, we compute the $y$-ified homology of $\PP_n^y$ in terms of the algebra $\cA_n$.

 \begin{lemma}
 \label{lem: homology colimit}
 (a) We have (for any $\k$):
 $$\displaystyle \HY(\PP_n^y)=\colim_{k\to \infty} \left(\Sigma_n^{-k}\HY(\FTy_n^k),\a_k\right).
 $$
 (b)The homology of $\PP_n^y$ is isomorphic to the (homogeneous) localization of $\cA_n$ in $\Delta_n$:
$$
\HY(\PP_n^y;\C)=\Span\left\{\frac{p}{\Delta_n^k}: p\in \cJ_n^k\right\}/\left(\frac{p}{\Delta_n^k}\sim \frac{p\Delta_n}{\Delta_n^{k+1}}\right)
$$
(c) Also,
$$
\HY(\PP_n^y;\C)=\C\left[\frac{p}{\Delta_n}: p\in \cJ_n\right]\subset \Frac\ \cA_n.
$$
 
\end{lemma}
We denote the localization of $\cA_n$ by $\cA_n[\Delta_n^{-1}]$. Note that $\cA_n[\Delta_n^{-1}]$ embeds into  $\Frac\ \cA_n$ by Lemma \ref{lem: not zero divisor}.

 \begin{proof}
 (a) This is standard but we recall the proof for completeness.  Consider the long exact sequence associated to the mapping cone $\HY(\PP_n^y) = \Cone\left(\HY(\mathbf{L})\buildrel A-\id \over \longrightarrow \HY(\mathbf{L})\right)$ as in the proof of Theorem \ref{thm:Pn as colimit}.  The map $\HY(\mathbf{L})\to \HY(\mathbf{L})$ is injective, since it is the identity map plus terms that increase the evident  filtration by rows in \eqref{eq:Py}.  Thus, we have a short exact sequence
 \[
0\to \bigoplus_{k\geq 0} \Sigma_n^{-k} \HY(\FTy_n^k) \buildrel A-\id\over \longrightarrow \bigoplus_{k\geq 0} \Sigma_n^{-k} \HY(\FTy_n^k) \to \HY(\PP_n^y)\to 0
\]
where the map $A-\id$ sends $x\in \HY(\FTy_n^k)$ to $\a_k(x)-x$. This shows that $\HY(\PP_n^y)\cong \operatorname{coker}(A-\id)$, which in turn is isomorphic to $\colim_{k\to \infty} (\Sigma_n^{-k}\HY(\FTy_n^k),\a_k)$. This completes the proof.

(b) By part (a) and Lemma \ref{lem: connecting map Delta} we have
\[
\HY(\PP_n^y;\C)=\colim_{k\to \infty}\left(
\begin{tikzcd}
\cJ_n^0 \arrow[r,"\Delta_n"] & \Sigma_n\inv \cJ_n^1 \arrow[r,"\Delta_n"]& \Sigma_n^{-2}\cJ_n^2\arrow[r,"\Delta_n"] & \cdots 
\end{tikzcd}
\right)
\]
By definition, this colimit is spanned by pairs $(p,k)$ for $p\in \cJ_n^k$ modulo equivalence relation $(p,k)\sim (\Delta_n p,k+1)$. We can
identify $(p,k)$ with $\frac{p}{\Delta_n^k}$, and this immediately implies (b).

For (c), we  observe that the graded algebra $\cA_n$ is generated in degrees 1 and 0. For $p_1,\ldots,p_k\in \cJ_n$ we get  $p_1\cdots p_k \in \cJ_n^k$ and
$$
\frac{p_1\cdots p_k}{\Delta_n^k}=\frac{p_1}{\Delta_n}\cdots \frac{p_k}{\Delta_n}.
$$
So $\cA_n[\Delta_n\inv]$ is generated (as a subalgebra of $\Frac\ \cA_n$) by the ratios $\frac{p}{\Delta_n},\ p\in \cJ_n$.
\end{proof}

\begin{lemma}
The isomorphism
$$
\HY(\PP_n^y)\simeq \cA_n[\Delta_n^{-1}]
$$
is compatible with the multiplication $\HY(\PP_n^y)\otimes \HY(\PP_n^y)\to \HY(\PP_n^y)$.
\end{lemma}

\begin{proof}
First, we would like to make the maps $\Sigma_n^{-k}\HY(\FTy^k_n)\to \HY(\PP_n^y)$ more explicit. For this, note that 
$$
\Hom_{\mathcal{Y}_{n,1}}(\Sigma_n^{-k}\FTy_n^k,\PP_n^y)\simeq \Hom_{\mathcal{Y}_{n,1}}(R,\Sigma_n^{k}\FTy_n^{-k}\star\PP_n^y)=\Hom_{\mathcal{Y}_{n,1}}(R,\PP_n^y)
$$
The second equation follows from Proposition \ref{prop:twisted projectors}. The map $\upsilon\in \Hom_{\mathcal{Y}_{n,1}}(R,\PP_n^y)$ then corresponds to a family of maps  
$$
\upsilon_k:\Sigma_n^{-k}\FTy_n^k\to \PP_n^y
$$
which are compatible with multiplication $(\Sigma_n^{-k}\FTy_n^k)\star (\Sigma_n^{-m}\FTy_n^m)\simeq \Sigma_n^{-k-m}\FTy_n^{k+m}$ on one side and $\PP_n^y\star \PP_n^y\to \PP_n^y$ from Theorem \ref{thm:map of diagrams} on the other.
Therefore we have an algebra homomorphism
$$
\bigoplus_k\Sigma_n^{-k}\HY(\FTy_n^k)\to \HY(\PP_n^y).
$$
Furthermore, it follows from the proof of Lemma \ref{prop:twisted projectors} that the map 
$$
\a_k\star \Id:\Sigma_n^{-k}\FTy_n^k\star \PP_n^y\to \Sigma_n^{-k-1}\FTy_n^{k+1}\star \PP_n^y 
$$
is a homotopy equivalence. Therefore the following triangle commutes up to homotopy:
$$
\begin{tikzcd}
\Sigma_n^{-k}\FTy_n^k \arrow{dr}{\upsilon_k}  \arrow{dd}{\a_k}& \\
 & \PP_n^y\\
\Sigma_n^{-k-1}\FTy_n^{k+1} \arrow[swap]{ur}{\upsilon_{k+1}}& 
\end{tikzcd}
$$
and $\upsilon_k$ define a degree zero chain map
$$
\upsilon_{\infty}:\hocolim_{k\to \infty}(\Sigma_n^{-k}\FTy_n^k,\a_k)\to \PP_n^y
$$
By Theorem \ref{thm:Pn as colimit} and Lemma \ref{lem:end Pn y}  the space of such maps is one-dimensional, so $\upsilon_n$ is homotopy equivalent to the map from Theorem \ref{thm:Pn as colimit} up to a scalar. By looking at the unit maps $\upsilon$ from $R$ to $\hocolim_{k\to \infty}(\Sigma_n^{-k}\FTy_n^k,\a_k)$ and $\PP_n^y$, we conclude that this scalar is 1 and the result follows.  
\end{proof}

Next, we compute the above localization $\cA_n[\Delta_n^{-1}]$  explicitly. 

\begin{theorem}
\label{thm: y-ified projector}
There is an algebra isomorphism
$$
\HY(\PP_n^y;\C)\cong \cA_n[\Delta_n\inv]\cong \C[y_1,\ldots,y_n,u_0,\ldots,u_{n-1},\xi_0,\ldots,\xi_{n-1}].
$$
The elements $u_i$, $\xi_i$  are uniquely characterized by the equations:
\begin{align}
\label{eq: interpolation u}
u_0+u_1y_i+u_2y_i^2+\ldots+u_{n-1}y_i^{n-1}&=x_i\ \mathrm{for\ all}\ i \\
\label{eq: interpolation xi}
 \xi_0+\xi_1y_i+\xi_2y_i^2+\ldots+\xi_{n-1}y_i^{n-1}&=\theta_i\ \mathrm{for\ all}\ i.
\end{align}
\end{theorem}
Note that \eqref{eq: interpolation u} and \eqref{eq: interpolation xi} imply
\begin{equation}
\label{eq: deg u and theta}
\wt(u_i)=qt^{-i}=Q^{2+2i}T^{-2i},\quad \wt(\xi_i)=at^{-i}=Q^{2i-2}T^{-2i}A.
\end{equation}

\begin{proof}
We abbreviate $\C[\yy,\uu,\xxi] =\C[y_1,\ldots,y_n,u_0,\ldots,u_{n-1},\xi_0,\ldots,\xi_{n-1}]$.
One can obtain $u_i$ and $\xi_i$ as elements of $\cA_n[\Delta_n\inv]$ by solving the equations \eqref{eq: interpolation u} and \eqref{eq: interpolation xi} using Cramer's Rule:
\begin{equation}
\label{eq: Cramer u}
u_i=\frac{\Alt\left(y_1^0y_2^1\cdots \widehat{y_i^{i-1}}\cdots y_n^{n-1}\cdot x_i\right)}{\Delta_n},
\end{equation}
\begin{equation}
\label{eq: Cramer xi}
\xi_i=\frac{\Alt\left(y_1^0y_2^1\cdots \widehat{y_i^{i-1}}\cdots y_n^{n-1}\cdot \theta_i\right)}{\Delta_n}.
\end{equation}
Clearly, the numerators in \eqref{eq: Cramer u} and \eqref{eq: Cramer xi} are antisymmetric, so $u_i$ and $\xi_i$ are well defined elements of the $\cA_n[\Delta_n\inv]$. This defines a homomorphism $\varphi: \C[\yy,\uu,\xxi]\to \cA_n[\Delta_n^{-1}]$.

Next, consider the homomorphism
$\varphi': \cA_n[\Delta_n^{-1}]\to \C[\yy,\uu,\xxi][\Delta_n^{-1}]$
defined by substituting $x_i$ and $\theta_i$ via
\eqref{eq: interpolation u} and \eqref{eq: interpolation xi}. We claim that the image of $\varphi'$ is actually contained in $\C[\yy,\uu,\xxi]$ and therefore $\varphi'$ is inverse to $\varphi$. 

Indeed, by Lemma \ref{lem: homology colimit}(c) the algebra $\cA_n[\Delta_n\inv]$ is generated by the elements $p/\Delta_n$ for $p\in \cJ_n$. Furthermore, we can write $p=\sum_k f_k(\xx,\yy,\ttheta)g_k(\xx,\yy,\ttheta)$ where $f_k$ are arbitrary polynomials and $g_k$ are antisymmetric polynomials, and
$$
\varphi'\left(\frac{p}{\Delta_n}\right)=\sum_k \varphi'(f_k(\xx,\yy,\ttheta))\frac{\varphi'(g_k(\xx,\yy,\ttheta))}{\Delta_n}.
$$
The  action of $S_n$ permutes $x_i,y_i$ and $\theta_i$ simultaneously, and fixes $u_i$ and $\xi_i$.  Therefore $\varphi'(g_k(\xx,\yy,\ttheta))$
 can be written as a polynomial in  $u_i$ and $\xi_i$ with coefficients being antisymmetric polynomials in $y_i$. In particular, it is divisible by $\Delta$ and $\varphi'(g_k(\xx,\yy,\ttheta))/\Delta_n.$ is a polynomial in  $u_i$ and $\xi_i$ with coefficients being
symmetric polynomials in $y_i$. We conclude that
$$
\varphi'\left(\frac{p}{\Delta_n}\right)\in \C[\yy,\uu,\xxi].
$$
\end{proof}

\begin{corollary}
\label{cor: u tilde comparison}
Let $\k$ be an arbitrary field of characteristic zero. Then 
$$
\HY(\PP_n^y;\k)\cong \k[y_1,\ldots,y_n,u_0,\ldots,u_{n-1},\xi_0,\ldots,\xi_{n-1}]
$$ where $u_i,\xi_i$ satisfy \eqref{eq: interpolation u} and \eqref{eq: interpolation xi}.
\end{corollary}

\begin{proof}
Without loss of generality we can assume $\k=\Q$. 
By Lemma \ref{lem:end Pn y} and the universal coefficients theorem we get
$$
\HY(\PP_n^y;\Q)\simeq \Q\left[\yy,\widetilde{u}_0,\ldots,\widetilde{u}_{n-1},\widetilde{\xi}_0,\dots,\widetilde{\xi}_{n-1},\right]
$$
By applying the universal coefficients theorem again , we can identify $$\HY(\PP_n^y;\C)\simeq \HY(\PP_n^y,\Q)\otimes \C,$$ and by Theorem \ref{thm: y-ified projector} get an embedding
$$
\Q\left[\yy,\widetilde{u}_0,\ldots,\widetilde{u}_{n-1},\widetilde{\xi}_0,\dots,\widetilde{\xi}_{n-1}\right]   \hookrightarrow
\C[\yy,u_0,\ldots,u_{n-1},\xi_0,\ldots,\xi_{n-1}] 
$$
which becomes an isomorphism after tensoring with $\C$. Note that the abstract generators $\widetilde{u}_i$ (resp. $\widetilde{\xi}_i$) do not need to agree with the explicit desired generators $u_i$ (resp. $\xi_i$), but they have the same respective weights.

Since $x_i$ belongs to $\HY(\PP_n^y;\Q)$, by \eqref{eq: basis HY} it can be written as a linear combination of $\widetilde{u}_j$ with coefficients in $\Q[\yy]$. The numerators of \eqref{eq: Cramer u} can then be also written as linear combinations of $\widetilde{u}_j$ with coefficients in $\Q[\yy]$. These coefficients are divisible by $\Delta_n$ over $\C$ and hence they are divisible by $\Delta_n$ over $\Q$ as well. So the element $u_i$ defined by \eqref{eq: Cramer u} is a linear combination of $\widetilde{u}_j$ with coefficients in $\Q[\yy]$. Similarly, $\xi_i$ defined by \eqref{eq: Cramer xi} a linear combination of $\widetilde{\xi}_j$ with coefficients in $\Q[\yy]$.

Here we used the following easy fact. Suppose $f,g\in \Q[z_1,\ldots,z_s]$ and there exists $h\in \C[z_1,\ldots,z_s]$ such that $f=gh$. Then $h\in \Q[z_1,\ldots,z_s]$. This can be proved by induction in $s$ using polynomial long division.

Going back to the proof, we can clarify the relation between $u_i$ and $\widetilde{u}_j$. By applying \eqref{eq: basis HY} again, we get
\begin{equation}
\label{eq: u tilde triangular}
u_i=p_{ii}\widetilde{u}_i+\sum_{j>i}p_{ij}(\yy)\widetilde{u}_j 
\end{equation}
where $p_{ij}$ are homogeneous polynomials in $\yy$ of degree $j-i$ with rational coefficients, in particular $p_{ii}\in \Q$. 

We claim that $p_{ii}\neq 0$ for all $i$. Suppose not, and $i$ is the largest index such that $p_{ii}=0$. Then by \eqref{eq: u tilde triangular} we have 
$$
u_i\in \Q[\yy,\widetilde{u}_{i+1},\ldots,\widetilde{u}_{n-1}]=\Q[\yy,u_{i+1},\ldots,u_{n-1}],
$$
so $u_j$ are not algebraically independent, contradiction.

Therefore $(p_{ij}(\yy))$ form an invertible triangular matrix and $\widetilde{u}_i$ can be written as polynomials in $u_j$ and $\yy$. Similarly, $\widetilde{\xi}_i$ can be written as polynomials in $\xi_j$ and $\yy$. We conclude that
$$
\HY(\PP_n^y;\Q)\cong \Q\left[\yy,\widetilde{u}_{0},\ldots,\widetilde{u}_{n-1},\widetilde{\xi}_{0},\ldots,\widetilde{\xi}_{n-1}\right]=\Q\left[\yy,u_{0},\ldots,u_{n-1},\xi_0,\ldots,\xi_{n-1}\right] 
$$
and the result follows.
\end{proof}

\subsection{Differentials in $y$-ified  homology}

The $y$-ified homology of $\one_n$ has a standard $\fgl_N$ differential  which is defined on generators (see Example \ref{ex: dN unknot}) by $$d_N(\theta_i)=x_i^N$$ and extended by Leibniz rule
$$
d_N(fg)=d_N(f)g+(-1)^{|f|}d_N(g).
$$ 
The differential $d_N$ has weight $q^{N}a\inv$, i.e.~it decreases $a$-degree by 1 and increases $q$-degree by $N$.  

\begin{lemma}
The differential $d_N$ preserves the ideal $\cJ_n$ and its powers $\cJ_n^k$.
\end{lemma}

\begin{proof}
For any $g\in \HY(\one_n;\C)=\C[\xx,\yy,\ttheta]$ and $w\in S_n$ we clearly have $d_N(w(g))=w(d_N(g))$. 

If $g$ is antisymmetric then $$w(d_N(g))=d_N(w(g))=d_N((-1)^{\sgn(w)} g)=(-1)^{\sgn(w)}d_N(g),$$ so $d_N(g)$ is also antisymmetric. Furthermore, for any $f$ and antisymmetric $g$ we have $d_N(fg)=d_N(f)g+fd_N(g)$, so $d_N(fg)\in \cJ$. We conclude that $d_N(\cJ_n)\subset \cJ_n$ and similarly $d_N(\cJ_n^k)\subset \cJ_n^k$. 
\end{proof}

\begin{lemma}
The $y$-ified Rasmussen $\fgl_N$ differential on $\HY(\FTy_n^k;\C)\simeq \cJ_n^k$ is given by the restriction of $d_N$ to $\cJ_n^k$. 
\end{lemma}

\begin{proof}
By Proposition \ref{prop:HH with dN is functor} the splitting map $\Psi$ commutes with $d_N$, and the result follows.  
\end{proof}

Next, we compute the action of the $y$-ified differential $d_N$ in $\HY(\PP_n^y;\k)$. 
First, we need some notations.  Define
$$
p(z)=\prod_{i=1}^{n}(z-y_i).
$$
This is a monic polynomial in $z$ of degree $n$ with coefficients in $\k[y_1,\ldots,y_n]$. Let $f(z)$ be a polynomial in $z$ with coefficients in $S[y_1,\ldots,y_n]$ where $S$ is some unique factorization domain (UFD) over $\k$. Using polynomial long division, we can divide $f(z)$ by $p(z)$ with remainder:
\begin{equation}
\label{eq: division with remainder}
f(z)=q(z)p(z)+r(z),\ \deg r(z)\le n-1.
\end{equation}
Such polynomials $q(z)$ and $r(z)$ are uniquely determined by $f(z)$ and have coefficients in $S[y_1,\ldots,y_n]$.
Below we will abbreviate \eqref{eq: division with remainder} by $r(z)=f(z)\mod p(z)$.

\begin{lemma}
\label{lem: remainder}
Suppose that $f(z)$ and $r(z)$ are as in \eqref{eq: division with remainder}. Then $r(z)$ is the unique polynomial of degree at most $n-1$ such that $r(y_i)=f(y_i)$ for all $i$.
\end{lemma}

\begin{proof}
If we substitute $z=y_i$ into \eqref{eq: division with remainder}, we get
$$
f(y_i)=q(y_i)p(y_i)+r(y_i)=r(y_i).
$$
Conversely, suppose that $h(z)\neq g(z)$,  $h(y_i)=f(y_i)$ and 
$\deg h(z)\le n-1$. Then $h(y_i)-g(y_i)=0$ for all $i$, so $h(z)-g(z)$ is divisible by $z-y_i$ for all $i=1,\ldots,n$. 

Since $S$ is a UFD, $S[y_1,\ldots,y_n,z]$ is also a UFD, and the polynomials $z-y_i$ are pairwise coprime for different $i$. Therefore $h(z)-g(z)$ is divisible by $\prod_{i=1}^{n}(z-y_i)=p(z)$ and $\deg (h(z)-g(z))\ge n$. Contradiction.
 \end{proof}

\begin{theorem}
\label{thm: y-ified differential}
Suppose $\k$ is a characteristic zero field.
The $y$-ified $\fgl_N$ Rasmussen spectral sequence for $\PP_n^y$ has $E_1$ page $$E_1=\HY(\PP_n^y;\k)\simeq \k[y_1,\ldots,y_n,u_0,\ldots,u_{n-1},\xi_0,\ldots,\xi_{n-1}]$$ and the first differential $d_N$ given by \eqref{eq: dN y-ified}. 
\begin{equation}
\label{eq: dN y-ified}
d_N(\xi(z))=u(z)^N \mod p(z)
\end{equation}
where 
$$
\xi(z)=\xi_0+\xi_1z+\xi_2z^2+\ldots+\xi_{n-1}z^{n-1}, u(z)=u_0+u_1z+u_2z^2+\ldots+u_{n-1}z^{n-1}
$$
and $p(z)=\prod_{i=1}^n (z-y_i)$ as above.

Furthermore, the homology on $E_2$ page is supported in $a$-degree zero and all higher differentials vanish, so that 
$$
E_2=E_{\infty}=\frac{\k[y_1,\ldots,y_n,u_0,\ldots,u_{n-1}]}{\left(d_N(\xi_0),\ldots,d_N(\xi_{n-1})\right)}
$$

\end{theorem}

\begin{proof}
First, we claim that the differential $d_N$ commutes with localization in $\Delta_n$ and satisfies the Leibniz rule. Indeed, this follows from $d_N(\Delta_n)=0$ and the fact that $d_N$ satisfies Leibniz rule on $\HY(\one_n)$. Note that 
$$
d_N\left(\frac{p}{\Delta_n^k}\right)=\frac{d_N(p)}{\Delta_n^k}
$$
for any $p\in \cJ_n^k$.

Next, we observe that $\cA_n$ and its localization $\cA_n[\Delta_n\inv]$ are graded (by the $a$-degree) so that $\theta_i,\xi_i$ have degree 1 and $x_i,y_i,u_i$ have degree zero. The differential $d_N$ decreases this degree by 1, so it sends $\xi_i$ to some polynomials in $y_i,u_i$. Also, $d_N(y_i)=d_N(u_i)=0$.

Now we use the equations \eqref{eq: interpolation u} and \eqref{eq: interpolation xi}. Let 
$$
g_N(z)=d_N(\xi_0)+d_N(\xi_1)z+d_N(\xi_2)z^2+\ldots+d_N(\xi_{n-1})z^{n-1}
$$
and 
$$
f_N(z)=(u_0+u_1z+u_2z^2+\ldots+u_{n-1}z^{n-1})^N.
$$
If we substitute $z=y_i$ for some $i=1,
\ldots,n$, we get
$$
g_N(z=y_i)=d_N(\xi_0)+d_N(\xi_1)y_i+d_N(\xi_2)y_i^2+\ldots+d_N(\xi_{n-1})y_i^{n-1}=
$$
$$
d_N\left(\xi_0+\xi_1y_i+\xi_2y_i^2+\ldots+\xi_{n-1}y_i^{n-1}\right)=d_N(\theta_i)=x_i^N
$$
while
$$
f_N(z=y_i)=(u_0+u_1y_i+u_2y_i^2+\ldots+u_{n-1}y_i^{n-1})^N=x_i^N.
$$
Note that here we used that $d_N$ commutes with $y_i$.
We conclude that $g_N(z=y_i)=f_N(z=y_i)$. Since $g_N(z)$ have $z$-degree at most $n-1$ by construction, by Lemma \ref{lem: remainder} we get $g_N(z)=f_N(z)\mod p(z)$. 

Let us prove that $E_2$ page is supported in $a$-degree zero. Observe that $(E_1,d_N)$ is the Koszul complex corresponding to $d_N(\xi_k)$, and by Lemma \ref{lem: regular sequence} below these form a regular sequence. Therefore all higher homology of the Koszul complex vanish. 
Finally, since $E_2$ page is supported in $a$-degree zero and all higher differentials decrease the $a$-degree, these must vanish and $E_2=E_{\infty}$.
\end{proof}

Below in Section \ref{sec: interpolation algebra} we give a more explicit description of the ideal $\left(d_N(\xi_0),\ldots,d_N(\xi_{n-1})\right)$.

\subsection{Back to the undeformed case}
 
 \begin{lemma}
 \label{lem: pi}
 Let $\k$ be a field of characteristic zero. We have $$\HHH(\PP_n;\k)\simeq \k\left[\widetilde{u}_0,\ldots,\widetilde{u}_{n-1},\widetilde{\xi}_0,\ldots,\widetilde{\xi}_{n-1}\right].$$ The natural forgetful map
 $\pi:\HY(\PP_n^y;\k)\to \HHH(\PP_n;\k)$ is a surjective algebra homomorphism and  sends the generator $u_i$ (resp. $\xi_i$) defined by \eqref{eq: interpolation u} (resp. \eqref{eq: interpolation xi}) to $\widetilde{u}_i$  (resp. $\widetilde{\xi}_i$) up to nonzero scalars.
 \end{lemma}

 \begin{proof}
 This follows from the proof of Corollary \ref{cor: u tilde comparison}. By Lemma \ref{lem: parity} we have $\HHH(\PP_n;\k)=\HY(\PP_n^y;\k)/(\yy)$.
 By \eqref{eq: u tilde triangular} we have $u_i=p_{ii}\widetilde{u_i}\mod \yy$ and $p_{ii}\neq 0$, so $\pi(u_i)=p_{ii}\widetilde{u_i}$. The proof for $\xi_i$ is similar.
 \end{proof}

From now on we will rescale the generators $\widetilde{u}_i,\widetilde{\xi}_i$ such that $\pi(u_i)=\widetilde{u}_i$ and $\pi(\xi_i)=\widetilde{\xi}_i$.

\begin{theorem}
\label{thm: HHH differential}
Let $\k$ be a field of characteristic zero.
There is a spectral sequence with $E_1$ page $$E_1=\HHH(\PP_n;\k)\simeq \k\left[\widetilde{u}_0,\ldots,\widetilde{u}_{n-1},\widetilde{\xi}_0,\ldots,\widetilde{\xi}_{n-1}\right]$$  and $E_2$ page is given by 
$$
E_2=H^*(E_1,d_N),\  d_N\left(\widetilde{\xi}(z)\right)=\widetilde{u}(z)^N \mod z^n.
$$
The differential $d_N$ satisfies Leibniz rule.
\end{theorem}

\begin{proof}
Since the differential $d_N$ commutes with the projection $\pi$ by \ref{prop:HH with dN is functor}, its action   on generators follows from Theorem \ref{thm: y-ified projector} and Lemma \ref{lem: pi}.
By Lemma \ref{lem: pi} $\pi$ is an algebra homomorphism which sends $y_i$ to zero, $u_i$ to $\widetilde{u}_i$ and $\xi_i$ to $\widetilde{\xi}_i$, so we get
$$
d_N\left(\widetilde{\xi}(z)\right)=d_N(\pi(\xi(z)))=\pi(d_N(\xi(z))).
$$
By Theorem \ref{thm: y-ified projector} $d_N(\xi(z))$ is the remainder of $u(z)^N$ modulo $p(z)=\prod (z-y_i)$. Since $p(z)$ is a monic polynomial, the coefficients of the remainder are obtained as certain polynomials in the coefficients of  $u(z)^N$ and $p(z)$, and we can apply $\pi$ to specialize all of these. Since $\pi$ is a homomorphism, we get that $\pi(d_N(\xi(z)))$ is the remainder of $\pi(u(z)^N)=\overline{u}(z)^N$ modulo $\pi(p(z))=\prod (z-0)=z^n$. 

It remains to prove that $d_N$ satisfies Leibniz rule, we use surjectivity of $\pi$: assume $f,g\in \HHH(\PP_n;\k)$. Then $f=\pi(f_1),g=\pi(g_1)$ for some $f_1,g_1\in \HY(\PP_n^y;\k)$. Now 
$$
d_N(fg)=d_N(\pi(f_1g_1))=\pi(d_N(f_1g_1))=\pi(d_N(f_1)g_1+f_1d(g_1))=
$$
$$
d_N(\pi(f_1))\pi(g_1)+\pi(f_1)d_N(\pi(g_1))=d_N(f)g+fd_N(g).
$$
Here we also used that $d_N$ commutes with $\pi$ and $\pi$ is a homomorphism. 
\end{proof}

\section{The interpolation algebra}
\label{sec: interpolation algebra}

\subsection{A formula for $y$-ified differential}

Recall that $p(z)=\prod_{i=1}^{n}(z-y_i)$. We give a closed formula for the coefficients of the remainder $r(z)$ in \eqref{eq: division with remainder}.

\begin{lemma}
\label{lem: remainder skew Schur}
Suppose that $r(z)=f(z)\mod p(z)$ as in \eqref{eq: division with remainder} and 
$$f(z)=\sum_{k=0}^{\deg\ f} f_kz^k,\ r(z)=\sum_{k=0}^{n-1} r_kz^k.$$
Then 
\begin{equation}
\label{eq: rk}
r_k=f_k+(-1)^{n-1-k}\sum_{j=n}^{\deg\ f}f_js_{(j-n+1,1^{n-1-k})}(\yy).
\end{equation}
where $s_{(j-n+1,1^{n-1-k})}(\yy)$ is the Schur polynomial in $y_1,\ldots,y_n$ for the hook diagram $(j-n+1,1^{n-1-k})$.
\end{lemma}

\begin{proof}
We have 
$$
r_0+r_1y_i+\ldots+r_{n-1}y_i^{n-1}=r(y_i)=f(y_i)
$$
which can be interpreted as a linear system of equations for $r_k$. By solving this by Cramer's Rule, we get
$$
r_k=\frac{(-1)^{k}}{\Delta}\det\left[f(y_i); 1;  y_i;\ldots; \widehat{y_i^k};\ldots; y_i^{n-1}\right]_{i=1}^{n}=
$$
$$
\frac{(-1)^{k}}{\Delta}\sum_{j=0}^{\deg\ f}f_j\det\left[ y_i^j; 1;  y_i;\ldots; \widehat{y_i^k};\ldots; y_i^{n-1}\right]_{i=1}^{n}.
$$ 
The  determinant 
$$
\det\left[ y_i^j; 1;  y_i;\ldots; \widehat{y_i^k};\ldots; y_i^{n-1}\right]_{i=1}^{n}
$$
equals $(-1)^{k}\Delta$ for $j=k$ and vanishes when $j\neq k$ and $0\le j\le n-1$. For $j\ge n$ we get
$$
\det\left[ y_i^j; 1;  y_i;\ldots; \widehat{y_i^k};\ldots; y_i^{n-1}\right]_{i=1}^{n}=
(-1)^{n-1} 
\det\left[1;  y_i;\ldots; \widehat{y_i^k};\ldots; y_i^{n-1},y_i^{j}\right]_{i=1}^{n}=
$$
$$
(-1)^{n-1} \Delta s_{(j-n+1,1^{n-1-k})}(\yy),
$$
and the result follows. 
\end{proof}

In what follows it will be useful to change coordinates and consider linear combinations of $r_k$ instead.  

\begin{lemma}
\label{lem: change of variables}
We have
$$
r_k+h_1(\yy)r_{k+1}+\ldots+h_{n-1-k}(\yy)r_{n-1}=\sum_{j=k}^{\infty}f_jh_{j-k}(\yy).
$$
\end{lemma}

\begin{proof}
We substitute $r_k$ using \eqref{eq: rk}. The first terms add up to
$$
f_k+h_1(\yy)f_{k+1}+\ldots+h_{n-1-k}(\yy)f_{n-1}=\sum_{j=k}^{n-1}f_jh_{j-k}(\yy).
$$\
For $j\ge n$ it suffices to compare the coefficients at $f_j$ and prove the identity (compare with \cite[Corollary 2.9]{HRW}):
\begin{equation}
\label{eq: hook schur to h}
\sum_{s=0}^{n-k-1}(-1)^{n-1-k-s}h_{s}(\yy)s_{(j-n+1,1^{n-1-k-s})}(\yy)=h_{j-k}(\yy).
\end{equation}
Note that we can write
$$
s_{(j-n+1,1^{n-1-k-s})}(\yy)=\sum_{t=0}^{n-1-k-s}(-1)^{t}h_{j-n+1+t}(\yy)e_{n-1-k-s-t}(\yy),
$$
so the left hand side of \eqref{eq: hook schur to h} can be rewritten as
$$
\sum_{s=0}^{n-k-1}\sum_{t=0}^{n-1-k-s}(-1)^{n-1-k-s-t}h_{s}(\yy)h_{j-n+1+t}(\yy)e_{n-1-k-s-t}(\yy).
$$
By changing the order of summation and using
$$
\sum_{s=0}^{n-k-1-t}(-1)^{n-1-k-s-t}h_{s}(\yy)e_{n-1-k-s-t}(\yy)=\begin{cases}
1 & \mathrm{if}\ n-1-k-t=0\\
0 & \mathrm{otherwise}\\
\end{cases}
$$
we are left with the term  
$$
h_{j-n+1+t}(\yy)|_{t=n-1-k}=h_{j-k}(\yy).
$$
\end{proof}

\begin{lemma}
\label{lem: d_N explicit}
For $n$ and $N$ arbitrary and $0\le k\le n-1$ we have
\begin{multline}
\label{eq: dN xi}
d_N(\xi_k)=\sum_{j_1+\ldots+j_N=k}u_{j_1}\cdots u_{j_N}+\\
(-1)^{n-1-k}\sum_{j=n}^{N(n-1)}\sum_{j_1+\ldots+j_{N}=j}u_{j_1}\cdots u_{j_N} s_{(j-n+1,1^{n-1-k})}(\yy).
\end{multline}
Furthermore, if 
\begin{equation}
\label{eq: zeta}
\zeta_k=\xi_k+h_1(\yy)\xi_{k+1}+\ldots+h_{n-1-k}(\yy)\xi_{n-1},\ 0\le k\le n-1
\end{equation}
then 
\begin{equation}
\label{eq: dN zeta}
d_N(\zeta_k)=\sum_{j=k}^{N(n-1)}\sum_{j_1+\ldots+j_{N}=j}u_{j_1}\cdots u_{j_N} h_{j-k}(\yy).
\end{equation}
\end{lemma}

\begin{proof}
This follows from Lemmas \ref{lem: remainder skew Schur} and \ref{lem: change of variables} applied to the polynomial
$$
f(z)=(u_0+u_1z+\ldots+u_{n-1}z^{n-1})^N=\sum_{j=0}^{N(n-1)}\sum_{j_1+\ldots+j_{N}=j}u_{j_1}\cdots u_{j_N}z^j.
$$
\end{proof}

\begin{example}
For $n=2$ and $N\ge 2$ we get
$$
d_N(\xi_0)=u_0^N-y_1y_2\sum_{j=2}^{N}\binom{N}{j}u_0^{N-j}u_1^{j}h_{j-2}(y_1,y_2),
$$
$$
d_N(\xi_1)=Nu_0^{N-1}u_1+\sum_{j=2}^{N}\binom{N}{j}u_0^{N-j}u_1^{j}h_{j-1}(y_1,y_2).
$$
Here we used that $s_{(j-1,1)}(y_1,y_2)=y_1y_2h_{j-2}(y_1,y_2)$.
\end{example}

\subsection{Specialization}

In this section we consider a certain specialization of $u_k$. 

\begin{definition}
Fix a number $M\gg 0$. We define the specialization homomorphism
$$
\varphi_M:\Q[u_0,\ldots,u_{n-1},y_1,\ldots,y_n]\to \Q[y_1,\ldots,y_n]
$$
$$
\varphi_M(u_k)=(-1)^{n-1-k}s_{(M-n+1,1^{n-1-k})}(\yy).
$$
\end{definition}

\begin{lemma}
\label{lem: specialization}
Under the homomorphism $\varphi_M$, we have the following specializations:
$$
\varphi_M(x_i)=y_i^M,\quad \varphi_M(d_N(\zeta_k))=h_{MN-k}(\yy).
$$
\end{lemma}

\begin{proof}
Let us apply Lemma \ref{lem: remainder skew Schur}   to $f(z)=z^M$. We get $r_k=\varphi_M(u_k)$, so
$$
\varphi_M(u(z))=z^M\mod p(z)
$$
and $\varphi_M(x_i)=\varphi_M(u(y_i))=y_i^M$. This proves the first equation. For the second equation, note that by the above
\begin{equation}
\label{eq: Nth powers agree}
\varphi_M(u(z)^N)\mod p(z)=z^{MN}\mod p(z).
\end{equation}
By applying Lemma \ref{lem: remainder skew Schur} to both sides of \eqref{eq: Nth powers agree} , we get
$$
\varphi_M(d_N(\xi_k))=(-1)^{n-1-k}s_{MN-n+1,1^{n-1-k}}(\yy).
$$
and by Lemma \ref{lem: change of variables} we get
$$
\varphi_M(d_N(\zeta_k))=h_{MN-k}(\yy).
$$
\end{proof}

\begin{lemma}
\label{lem: regular sequence}
a) The polynomials $d_N(\zeta_k)$ defined by \eqref{eq: dN zeta} form a regular sequence. 

b) The polynomials $d_N(\xi_k)$ defined by \eqref{eq: dN xi} form a regular sequence.
\end{lemma}

\begin{proof}
Consider a longer sequence 
\begin{equation}
\label{eq: long sequence}
u_0-\varphi_M(u_0),\ldots,u_{n-1}-\varphi_M(u_{n-1}),d_N(\zeta_0),\ldots,d_N(\zeta_{n-1}).
\end{equation}
Clearly, $u_0-\varphi_M(u_0),\ldots,u_{n-1}-\varphi_M(u_{n-1})$ form a regular sequence, let $I$  denote the ideal generated by these polynomials. Then 
$$
\Q[u_0,\ldots,u_{n-1},y_1,\ldots,y_{n}]/I\simeq \Q[y_1,\ldots,y_n]
$$
and for all polynomials $g\in \Q[u_0,\ldots,u_{n-1},y_1,\ldots,y_{n}]$ we have $g=\varphi_M(g)\mod I$. In particular, by Lemma \ref{lem: specialization} we get
$$
d_N(\zeta_k)=\varphi_M(d_N(\zeta_k))=h_{MN-k}(\yy)\mod I.
$$
Since $h_{MN}(\yy),h_{MN-1}(\yy),\ldots,h_{MN-n+1}(\yy)$ form a regular sequence in $\Q[y_1,\ldots,y_n]$, we conclude that $d_N(\zeta_0),\ldots,d_N(\zeta_{n-1})$ form a regular sequence modulo $I$, and the sequence \eqref{eq: long sequence} is regular.

Next, we define a grading on $\Q[u_0,\ldots,u_{n-1},y_1,\ldots,y_{n}]$ such that $\deg(y_i)=1$ and $\deg(u_k)=M-k$. The polynomials $u_k-\varphi_M(u_k)$ are homogeneous with respect to this degree. The polynomials $d_N(\zeta_k)$ are also homogeneous since every term in \eqref{eq: dN zeta} has degree
$$
(M-j_1)+\ldots+(M-j_N)+(j_1+\ldots+j_N-k)=MN-k.
$$
It is well known that whenever a sequence of homogeneous polynomials in a (positively) graded ring is regular in one order, it is regular in any other order, and therefore any subsequence is regular. We conclude that $d_N(\zeta_0),\ldots,d_N(\zeta_{n-1})$ is a regular sequence. Since the $\xi_k$ are related to the $\zeta_k$ by a unitriangular matrix, the sequence $d_N(\xi_0), \dots,d_N(\xi_{n-1})$ is also regular and generates the same ideal.
\end{proof}

\subsection{Divided differences}

The symmetric group $S_n$ acts on $\k[\yy,\uu,\xxi]$ by permuting the $y_i$ and leaving the $u_k$ and $\xi_k$ fixed, and this action commutes with both the differential $d_N$ and the specialization map $\varphi_M$ from the previous section. Therefore, for $s_i = (i \quad i+1)$ a transposition, the divided difference operator
\begin{equation*}
    \partial_i: f \mapsto (f - s_if) / (y_i - y_{i+1})
\end{equation*}
acts on $\k[\yy,\uu,\xxi]$, commuting with $d_N$ and $\varphi_M$, since the coefficient in $\yy$ of any monomial in $\uu$ and $\xxi$ of $f - s_if$ will be divisible by $y_i - y_{i+1}$.

These operators provide another route to the above calculations. If we set $\partial_* = \partial_{n-1}\cdots\partial_2\partial_1$, then $\partial_* y_1^k = h_{k-n+1}(\yy)$, so the elements $\zeta_k$ from above may be expressed as
\begin{equation*}
    \zeta_k = \sum_i h_i(\yy) \xi_{k+i} = \partial_* \sum_i y_1^{k+n-1} \xi_{k+i} = \partial_* \left(y_1^{n-1-k} \theta_1 \right)
\end{equation*}
Therefore, $$d_N \zeta_k = \partial_*\left(y_1^{n-1-k} d_N(\theta_1) \right) = \partial_*\left(y_1^{n-1-k} x_1^N \right).$$ This immediately gives the specialization $$\varphi_M (d_N \zeta_k) = \partial_*\left(y_1^{n-1-k} \cdot y_1^{MN} \right) = h_{MN - k}(\yy),$$ and a slightly more involved calculation recovers equation~\eqref{eq: dN zeta} as well.

\subsection{Coproduct}

As a side remark, we can define a curious coproduct on $\HY(\PP_n^y)$ over $\k[\yy]$. Suppose that we have two interpolation problems
$$
u_0+u_1y_i+\ldots+u_{n-1}y_i^{n-1}=x_i,\ \widetilde{u}_0+\widetilde{u}_1y_i+\ldots+\widetilde{u}_{n-1}y_i^{n-1}=\widetilde{x_i}
$$
and we want to solve the interpolation problem 
$$
v_0+v_1y_i+\ldots+v_{n-1}y_i^{n-1}=x_i\widetilde{x_i}.
$$
Then we can write 
\begin{equation}
\label{eq: coproduct}
v(z)=u(z)\widetilde{u}(z)\mod p(z)
\end{equation}
and this uniquely determines $v_k$ as polynomials of $u_i,\widetilde{u}_i,y_i$. By Lemma \ref{lem: remainder skew Schur} we get
\begin{equation}
v_k=\sum_{i+j=k}u_i\widetilde{u}_{j}+(-1)^{n-k-1}\sum_{i+j\ge n}u_i\widetilde{u_j}s_{(i+j-n+1,1^{n-1-k})}(y_1,\ldots,y_n).
\end{equation}
We can interpret \eqref{eq: coproduct} as a coproduct
$$
\Delta:\HY(\PP_n^y)\to \HY(\PP_n^y)\otimes_{\k[\yy]}\HY(\PP_n^y),\  \Delta(u_k)=v_k(u_i,\widetilde{u}_i,y_i).
$$
It is easy to see that $\Delta$ is cocommutative and coassociative.

\section{Further extensions}

\subsection{An algebraic conjecture}
\label{sec: algebraic conjecture}

In this section we introduce some classes in homology of $d_N$. See \cite[section 3.3]{GOR} and \cite[section 2.2]{GL}. Recall that 
$$
d_N(\xi(z))=u(z)^N\mod z^n.
$$

\begin{definition}
For $k=1,\ldots,n-1$ we define 
$$
\mu_k=\sum_{i+j=k}(Ni-j)u_i\xi_j .
$$
\end{definition}

\begin{lemma}
We have $d_N(\mu_k)=0$.
\end{lemma}

\begin{proof}
We consider the generating functions
$$
\xi(z)=\sum_{i=0}^{n-1}\xi_iz^i,\ u(z)=\sum_{i=0}^{n-1}u_iz^i,
\dot{\xi}(z)=\sum_{i=1}^{n-1}i\xi_iz^{i-1},\ 
\dot{u}(z)=\sum_{i=1}^{n-1}iu_iz^{i-1}
$$
We have
$$
d_N(\xi(z))=u(z)^N\mod z^n,\ d_N(\dot{\xi}(z))=N\dot{u}(z)u(z)^{N-1}\mod z^{n-1}.
$$
Furthermore, we can define $\mu(z)=\sum_{k=1}^{n-1}\mu_kz^{k-1}$, then
$$
\mu(z)=\sum_{i,j}(Ni-j)u_i\xi_j  z^{i+j-1}\mod z^{n-1}=$$
$$N\sum_{i,j}iu_i z^{i-1}\cdot \xi_j z^j-\sum_{i,j}u_i z^{i}\cdot j\xi_j z^{j-1}\mod z^{n-1}=
$$
$$
N\dot{u}(z)\xi(z)-u(z)\dot{\xi}(z)\mod z^{n-1}.
$$
Therefore
$$
d_N(\mu(z))=N\dot{u}(z)u(z)^N-u(z) N\dot{u}(z)u(z)^{N-1}=0\mod z^{n-1}.
$$
\end{proof}
\begin{example}
We have $\mu_1=Nu_1\xi_0-u_0\xi_1$. Since
$$
d_N(\xi_0)=u_0^{N},\ d_N(\xi_1)=Nu_0^{N-1}u_1
$$
we get $d_N(\mu_0)=0$.
\end{example}

\begin{conjecture}
\label{conj: mu}
The homology of $d_N$ is generated by $u_i$ and $\mu_k$ as an algebra.
\end{conjecture}

For $N=2$ and $n\le 8$ this conjecture has been extensively verified  in \cite{GOR}, but remains open in general.

\begin{theorem}
\label{thm: multiplicative collapse}
Assuming Conjecture \ref{conj: mu},  Rasmussen spectral sequence for $sl(N)$ homology of $T(n,\infty)$ collapses at the $E_2$ page.
\end{theorem}

\begin{proof}
We recall that a spectral sequence $(E^{p,q}_r,d_r)$ is \emph{multiplicative} if it satisfies the following:
\begin{itemize}
\item For each $r$ there is a collection of bigraded  products
$$
E^{p,q}_{r}\otimes E^{s,t}_{r}\to E^{p+s,q+t}_{r}
$$
In other words, the $r$-th page $E^{\bullet,\bullet}_{r}$ has a bigraded algebra structure.
\item The differential $d_{r}$ on $E^{\bullet,\bullet}_{r}$ satisfies the Leibniz rule
\item The induced multiplication on $H^*(E^{\bullet,\bullet}_{r},d_r)=E^{\bullet,\bullet}_{r+1}$ agrees with the algebra structure on $E^{\bullet,\bullet}_{r+1}$.
\end{itemize}
Let $d_N^{(r)}$ denote the $r$-th differential in the Rasmussen spectral sequence, with $d_N^{(1)}=d_N$. 
By Proposition \ref{prop: braid dN multiplicativity} the spectral sequence is multiplicative, and the differentials $d_N^{(r)}$ satisfy Leibniz rule. 

By Conjecture \ref{conj: mu} the $E_2$ page of the spectral sequence is generated by $u_k$ (which have $a$-degree 0) and $\mu_k$ (which have $a$-degree 1). The differential $d_N^{(r)}$ decreases the $a$-degree by $r$ and therefore 
$$
d_N^{(r)}(u_k)=d_N^{(r)}(\mu_k)=0
$$
and $d_N^{(r)}=0$. We conclude that the spectral sequence collapses at the $E_2$ page.
\end{proof}

\subsection{More general potentials}

Let $W\in \k[x]$ be an arbitrary one-variable polynomial called the potential, and let $\partial W$ denote its derivative with respect to $x$. One can construct a more general Khovanov--Rozansky link homology theory depending on $W$, with $W(x)=\frac{1}{N+1}x^{N+1}$ corresponding to $\gll_N$ homology. The construction runs as follows.

Let $M\in \SBim_n$ be a Soergel bimodule, and $F^{\bullet}$ is free resolution. As in Lemma \ref{lemma:res and h} we can find $d_{\partial W}\in \End(F^{\bullet})$ such that 
\begin{equation}
\label{eq: dW commutator}
[d_F,d_{\partial W}]=\sum_{i=1}^{n}\left(W(x_i)-W(x'_i)\right),
\end{equation}
since the right hand side is symmetric in $\xx,\xx'$ and vanishes on $M$. Furthermore $d_{\partial W}\circ d_{\partial W}$ is closed and exact. The right hand side of \eqref{eq: dW commutator} vanishes on $F^{\bullet}\otimes_{\k[\xx,\xx']}R$ and defines a differential $d_{\partial W}$ on $\HH(M)$. 

In the Koszul complex $\Delta_n(\xx,\xx')$ we have
$$
d_{\partial W}(\theta_i)=\frac{W(x_i)-W(x'_i)}{x_i-x'_i}
$$
which becomes $\partial W(x_i)$ after specializing $x_i=x'_i$. This explains the notations for the differential.

One can replace $d_N$ by $d_{\partial W}$  rest of the the constructions in Section \ref{ss: Rasmussen} extend verbatim and allow one to define link homology theory $H_{\partial W}$ (resp. $\HY_{\partial W}$), and a spectral sequence from $\HHH$ (resp. $\HY$) to it.
We can now generalize Theorem \ref{thm: y-ified differential} to this setting.

\begin{theorem}
\label{thm: y-ified W differential}
The differential $d_{\partial W}$ on $\HY(\PP_n^y)$ is given by the generating function
\begin{equation}
\label{eq: dN y-ified W}
d_{\partial W}(\xi(z))=\partial W(u(z)) \mod p(z).
\end{equation} 
\end{theorem}

\begin{proof}
As in the proof of Theorem \ref{thm: y-ified differential}, we note that $d_{\partial W}(y_i)=0$, so $d_{\partial W}(\Delta_n)=0$ and $d_{\partial W}$ commutes with localization in $\Delta_n$.

Furthermore, $d_{\partial W}(\xi_i)$ is a polynomial in $y_i,u_i$, and we define
$$
g_{\partial W}(z)=d_{\partial W}(\xi(z))=d_{\partial W}(\xi_0)+d_{\partial W}(\xi_1)z+\ldots+d_{\partial W}(\xi_{n-1})z^{n-1}.
$$
This is a polynomial in $z$ of degree at most $n-1$, and its value at $y_i$ equals
$$
g_{\partial W}(z=y_i)=
d_{\partial W}(\xi(z))=d_{\partial W}(\xi_0)+d_{\partial W}(\xi_1)y_i+\ldots+d_{\partial W}(\xi_{n-1})y_i^{n-1}=
$$
$$
d_{\partial W}(\xi_0+\xi_1y_i+\ldots+\xi_{n-1}y_i^{n-1})=d_{\partial W}(\theta_i)=\partial W(x_i)=\partial W(u_0+u_1y_i+\ldots+u_{n-1}y_i^{n-1}).
$$
Therefore by Lemma \ref{lem: remainder} we get
$$
g_{\partial W}(z)=\partial W(u(z))\mod p(z).
$$
\end{proof}

By specializing $y_i=0$ we get the following.

\begin{theorem}
\label{thm: HHH W differential}
There is a spectral sequence with $E_1$ page $$E_1=\HHH(\PP_n)\simeq \C[ u_0,\ldots, u_{n-1},\xi_0,\ldots,\xi_{n-1}]$$  and $E_2$ page is given by 
$$
E_2=H^*(E_1,d_{\partial W}),\  d_{\partial W}( \xi(z))=\partial W(u(z)) \mod z^n.
$$
The differential $d_{\partial W}$ satisfies Leibniz rule.
\end{theorem}

To find the analogue of $\mu(z)$ from Section \ref{sec: algebraic conjecture}, we define a polynomial 
$$H(x)=\mathrm{GCD}(\partial W(x),\partial^2 W(x)).$$
Concretely, if $\partial W(x)=\prod_{i=1}^{s} (x-\lambda_i)^{m_i}$ then $$H(x)=\prod_{m_i\ge 2} (x-\lambda_i)^{m_i-1}.$$
If all roots $\lambda_i$ of $\partial W(x)$ are simple and $m_i=1$ then we set $H(x)=1$.

\begin{lemma}
The coefficients of the following series are in the kernel of $d_{\partial W}$:
\begin{equation}
\mu^{W}(z)=\sum_{i=0}^{n-1}\mu^W_iz^i=\frac{\partial^2 W(u(z))}{H(u(z))}\dot{u}(z)\xi(z)-\frac{\partial W(u(z))}{H(u(z))}\dot{\xi}(z)\mod z^n.
\end{equation}
\end{lemma}

Note that by construction $H(x)$ divides both $\partial^2W(x)$ and $\partial W(x)$, so $\mu^{W}(z)$ is well defined.

\begin{proof}
Indeed, we have $d_{\partial W}(\xi(z))=\partial W(u(z))$, therefore $d_{\partial W}(\dot{\xi}(z))=\partial^2 W(u(z))\dot{u}(z)$. Now 
$$
d_{\partial W}(\mu^{W}(z))=d_{\partial W}\left(\frac{\partial^{2}W(u(z))}{H(u(z))}\dot{u}(z)\xi(z)-\frac{\partial W(u(z))}{H(u(z))}\dot{\xi}(z)\right)=
$$
$$
\frac{\partial^2 W(u(z))}{H(u(z))}\dot{u}(z)\partial W(u(z))-\frac{\partial W(u(z))}{H(u(z))}\partial^2 W(u(z))\dot{u}(z),
$$
and the result follows.
\end{proof}

\begin{conjecture}
\label{conj: mu W}
The homology of $d_W$ is generated by $u_i$ and $\mu_k^W$ as an algebra.
\end{conjecture}

We conclude the following.

\begin{theorem}
Assuming Conjecture \ref{conj: mu W},  Rasmussen spectral sequence for Khovanov--Rozansky homology of $T(n,\infty)$ with potential $W$ collapses at the $E_2$ page.
\end{theorem}

\end{document}